\documentclass[11pt,a4paper]{article}

\usepackage{latexsym,amsfonts,amssymb,setspace,galois,amsmath,color,mathrsfs, amstext,bm,bookmark,enumerate}
\usepackage{amsbsy, amsopn, amscd, amsxtra, amsthm,authblk,indentfirst}
\usepackage{graphicx,epsfig,subfigure}
\usepackage[titletoc]{appendix}
\usepackage{hyperref,cleveref,upref}
\usepackage{geometry}
\usepackage[authoryear,numbers,round,sort&compress,colon,super]{natbib}
\bibpunct{[}{]}{,}{n}{}{,} 
\usepackage{float}
\usepackage{yhmath}
\usepackage{microtype}

\newtheorem{theorem}{Theorem}[section]

\newtheorem{lemma}[theorem]{Lemma}
\newtheorem{proposition}{Proposition}

\theoremstyle{definition}
\newtheorem{definition}[theorem]{Definition}
\newtheorem{remark}{Remark}

\date{}                                     
\title{\bf{Spatiotemporal attractors generated by the Turing-Hopf bifurcation in a time-delayed reaction-diffusion system}\thanks {Supported by the National Natural Science Foundation of China (No.11371112).}}

\author[a]{Qi An\thanks {
		E-mail address: anqi@stu.hit.edu.cn.}}
\author[a]{Weihua Jiang\thanks {Corresponding author. E-mail address: jiangwh@hit.edu.cn. }}

\affil[a]{Department of Mathematics, Harbin Institute of Technology, Harbin, 150001, P.R. China.}

\begin{document}
\maketitle
\begin{abstract}
We study the Turing-Hopf bifurcation and give a simple and explicit calculation formula of the normal forms for a general two-components system of reaction-diffusion equation with time delays. We declare that our formula can be automated by Matlab. At first, we extend the normal forms method given by Faria in 2000 to Hopf-zero singularity. Then,
an explicit formula of the normal forms for Turing-Hopf bifurcation are given.  Finally, we obtain the possible attractors of the original system near the Turing-Hopf singularity by the further analysis of the normal forms, which mainly include, the spatially non-homogeneous  steady-state solutions, periodic solutions and  quasi-periodic solutions.
\end{abstract}

\section{Introduction}
Partial functional differential equations (PFDEs) are often used as  mathematical models to describe various phenomena in population ecology, biology and chemistry by many researchers 
{(see e.g. \cite{ Chen2017,Hadeler2012Interaction,Wu1996Theory,Zhao2003Dynamical} and references therein)}. 
The theory of PFDEs had been extensively studied since 1970s, and the
main published work included 
existence, uniqueness regularity and  stability of solutions, comparison theorem, upper and lower solutions, traveling wave solutions, bifurcation theory, centre manifolds and normal forms theory {(see e.g. \cite{Faria2000Normal,Li2006Existence,Lin1992Centre,martin1991reaction,pao1996dynamics,pao2002convergence,Travis1976Existence,Wu2001Traveling})}.

Patterns are ubiquitous in nature and always have certain regularities in space or time, such as cracks on the trunks, patterns of the animal's epidermis and structures of the desert surface. The importance of formation and development of patterns is self-evident. In 1952, Turing \cite{Turing1952} proposed a classical mechanism to  explain the formation of spatially heterogeneous patterns. He suggested that `morphogens' through the reaction and diffusion with other substances give a feedback control to the morphogenesis. From a mathematical point of view, the patterns are the stationary solutions of the reaction diffusion systems, which are results from the steady-state bifurcation. In the later researches and experiments, such Turing structures have been confirmed and widely applied to explain the main phenomena of morphogenesis (see e.g. \cite{Castets1990Experimental,Lengyel1991Modeling,Murray2003II,Ouyang1991Transition,Segel1972Dissipative}). 
In this paper, we will focus on another spatio-temporal pattern, which is oscillate in both space and time. This types of patterns had been mentioned in {\cite{Maini1997Spatial,Rovinsky1992Interaction,Shi2015Spatial}}. 
And Turing-Hopf bifurcation, which is one types of Hopf-zero bifurcation, is considered as one of the mechanisms to generate them.

As is well known, diffusion could lead to the emerge of Turing patterns, while time delays always cause periodic oscillations in dynamical systems. Therefore, to study the the Hopf-zero (Turing-Hopf) bifurcation of the PFDEs  is significant. In this work, we consider
a general reaction-diffusion system with a discrete delay consisting of two equations subject to Neumann boundary conditions and defined on a one-dimensional spatial domain $\Omega=(0,l\pi)$, with $l\in \mathbb{R^{+}}$.
\begin{equation}\label{A}
\left\{
\begin{aligned}
&\frac{du(t)}{dt}\!-\!d_{1}(\mu)\Delta u(t)\!=\!f(\mu,u(t),v(t),u(t\!-\!1),v(t\!-\!1)),\; x\in(0,l\pi), ~t>0,\\
&\frac{dv(t)}{dt}\!-\!d_{2}(\mu)\Delta v(u)\!=\!g(\mu,u(t),v(t),u(t\!-\!1),v(t\!-\!1)),\; x\in(0,l\pi), ~t>0,\\
&u_x=v_x=0, \hspace{6.05cm}x=0~\mathrm{or}~l\pi, ~t>0,\\
&u(x,s)=\varphi(x,s), v(x,s)=\psi(x,s),\hspace{2.2cm} x\in(0,l\pi), ~s\in[-1,0].
\end{aligned} \right.
\end{equation}
Here $\mu =(\mu_1,\mu_2)\in{\mathbb{R}}^{2},$ $d_{1}, d_{2}\in C^{1} ({\mathbb{R}}^{2},\mathbb{R}_{+}),$ $f, g\in C^{1,k}({\mathbb{R}}^{2}\times {\mathbb{R}}^{4},\mathbb{R})$  with $k\geq3$ and satisfy $f(\mu,0,0)=g(\mu,0,0)=0 ,$ for all $ \mu\in{\mathbb{R}}^{2}.$

In fact, the patterns which are formed by Hopf-zero bifurcation of other types of equations has been extensively studied. Such as the periodic or quasi-periodic oscillations in retarded function differential equations (RFDEs) (see e.g. \cite{Wang2010Hopf}), the spatially inhomogeneous coexisting steady states and periodic solutions in PDEs (see e.g. \cite{Baurmann2007Instabilities,De1996Spatiotemporal,Yang2016Spatial}). 
To the best of our knowledge, the results in PFDEs are very limited.

Center manifold and the normal forms theory are important ways to study the bifurcation of nonlinear systems. For PFDEs, Faria \cite{Faria2000Normal} proposed an excellent method to compute the normal forms based on the center manifold theory introduced in \cite{Lin1992Centre}. 
This normal forms method had been widely applied in Hopf bifurcation, and it plays a key role in the existence and stability of the spatial periodic solutions (see e.g. \cite{huang2017bifurcation,xu2017bifurcation}).  

The purpose of this paper is to extend the normal forms method given by Faria and Magalhaes \cite{Faria2000Normal,Faria1995Normal} to Hopf-zero bifurcation with perturbation parameters in PFDEs. And then give the explicit calculation formula of the normal forms for the Turing-Hopf bifurcation as simple as possible. 

Firstly, we present the derivation process of the  decomposition of phase space and center manifold reduction, then give an abstract formula of the normal forms near the Hopf-zero bifurcation point of \eqref{A} by extend the method proposed in \cite{Faria2000Normal,Faria1995Normal,Wu1996Theory}. 
Secondly, we applied the abstract formula to Turing-Hopf bifurcation and obtain the exact calculation formula of the normal forms. Finally, we show that the dynamics of \eqref{A} can be governed by a 3-dimensions ordinary differential system, with the parameters can be expressed by those parameters in original PFDEs. By studying the corresponding 2-dimensions amplitude system, the dynamical properties of the original system can be obtained.  Mainly include, the existence of 
\begin{itemize}
	\item the coexisting spatially non-homogeneous steady-state solutions,
	\item the coexisting spatially non-homogeneous periodic solutions,
	\item the coexisting  spatially non-homogeneous quasi-periodic solutions.
\end{itemize} 
The advantages of our calculation formula can be conclude as, 
\begin{enumerate}[(1)]
	\item it can be applied directly to PDEs,
	\item it can be extend to PFDEs with other boundary conditions or higher spatial domains by reanalyze the eigenvalue problem,
	\item  the entire computing process can be fully automated after getting the  bifurcation points. In other words, it is suitable for the computer calculation.
\end{enumerate}

The paper is organized as following: In Section 2, eigenvalue problems, decomposition of the phase space, center manifold reduction and the normal forms for Hopf-zero bifurcation of \eqref{A} are presented.
In Section 3, giving the precise eigenfunctions and the exact calculation formula of the normal forms to Turing-Hopf bifurcation. We claim that Hopf-pitchfork bifurcation is arising.
In Section 4, the dynamic behaviors of the original system have been studied through the 3-orders truncated normal forms and the corresponding amplitude system. In Section 5, taking Holling-Tanner models as an example to present the calculation process of the normal forms. Some numerical simulations of spatiotemporal patterns and spatial patterns are also demonstrated. Finally, the paper ends with some conclusions. The formulas of some coefficient vectors that mentioned in Section 3 are given in Appendix A.

\section{Reduction and the normal forms for PFDEs with Hopf-zero sigularity}
This section is divided into three parts. The first part is the eigenvalue problem for system \eqref{A}, the second part is the decomposition of phase space, the third part is the calculation of the normal forms near a Hopf-zero bifurcation point.

First of all,
define a real-valued Hilbert space $$X:=\{(u,v)\in{H^{2}(0,l\pi)\times H^{2}(0,l\pi)} : (u_{x},v_{x})|_{x=0,l\pi}=0\}.$$
The corresponding complexification is $X_{\mathbb{C}}:=\{ x_{1}+ix_{2} : x_{1},x_{2}\in X \}$, with the complex-valued $L^2$ inner product 
$$\langle U_{1},U_{2}\rangle=\int_{0}^{l\pi}(\overline{u}_{1}u_{2}+\overline{v}_{1}v_{2})dx,\quad\mathrm{for}\;\; U_{i}=(u_{i},v_{i})\in X_{\mathbb{C}}.$$
Let 
$\mathcal{C}:=C([-1,0],X_{\mathbb{C}})$ denote the phase space
with the sup norm. We write $\varphi^{t}\in \mathcal{C}$ for  $\varphi^{t}(\theta)=\varphi(t+\theta),-1\leq \theta \leq0.$

The system \eqref{A} can be rewritten as an abstract equation in phase space $\mathcal{C}$:
\begin{equation}\label{eqB}
\frac{d}{dt}U(t)=D(\mu)\Delta U(t)+L(\mu)(U^{t})+F(\mu,U^{t}).
\end{equation}
Here $D(\mu)=\mathrm{diag}~(d_{1}(\mu),d_{2}(\mu))$,~ $U=(u,v)^{\mathrm{T}}\in X_{\mathbb{C}}$,~
$U^{t}=(u^{t},v^{t})^{\mathrm{T}}\in \mathcal{C}$.
$L:{\mathbb{R}}^{2}\times\mathcal{C}\rightarrow X_{\mathbb{C}} $ is a bounded linear operator.
And $F:{\mathbb{R}}^{2}\times\mathcal{C}\rightarrow X_{\mathbb{C}} $ is a $C^{k}$ $(k\geq3)$ function and satisfies $F(\mu,0)=0,$ $D_{\varphi}F(\mu,0)=0$.

The linearized equation of \eqref{eqB} at $(0,0)$ is:
\begin{equation}\label{eqC}
\frac{d}{dt}U(t)=D(\mu)\Delta U(t)+L(\mu)(U^{t}).
\end{equation}

\subsection{Eigenvalues problem}\label{eigenvalues}

It is well known that the eigenvalue problem $$\varphi''(x)=\lambda\varphi(x),\quad x\in (0,l\pi),\quad \varphi'(0)=\varphi'(l\pi)=0,$$ has eigenvalues $\lambda_{n}=-\frac{n^{2}}{l^{2}}(n=0,1,2,\cdots),$ with corresponding eigenfunctions
\begin{equation*}
\beta_{n}(x)=\frac{\cos{\frac{n}{l}}x} {\|\cos{\frac{n}{l}}x\|_{L^{2}}}=
\left\{\begin{aligned}
&\sqrt{\frac{1}{l\pi}},& &n=0,&\\
&\sqrt{\frac{2}{l\pi}}\cos{\frac{n}{l}}x,& &n\geq 1.&\\
\end{aligned}       \right.
\end{equation*}
Let $\varphi_{n}^{(1)}=(\beta_{n},0)^{\mathrm{T}},$ $\varphi_{n}^{(2)}=(0,\beta_{n})^{\mathrm{T}}.$
Then $\{\varphi_{n}^{(i)}\}_{n\geq 0}$ are eigenfunctions of $D(\mu)\Delta$ with corresponding eigenvalues $\lambda_{n}^{(i)}(\mu)=-d_{i}(\mu)\frac{n^{2}}{l^{2}}$ $(i=1,2).$ And $\{\varphi_{n}^{(i)}\}_{n\geq 0}$ form an orthonormal basis of $X.$
\begin{definition}
	For $\lambda=\lambda(\mu)$, it is called a characteristic value of \eqref{eqC} if there exists a
	\begin{displaymath}
	y={\sum_{n=0}^{\infty}{a_{n}\varphi_{n}^{(1)}+b_{n}\varphi_{n}^{(2)}}}={\sum_{n=0}^{\infty} {\left(\begin{array}{c}
			a_{n}\\ b_{n}
			\end{array}\right)}\beta_{n}}
	\in {{dom(\Delta)} \backslash {\{0\}}}
	\end{displaymath}
	solving the characteristic equation
	\begin{equation}\label{eqD}
	\mathbf{\Delta}(\lambda)y=\lambda Iy-D(\mu)\Delta y-L(\mu)(e^{\lambda \cdot}y)=0. 
	\end{equation}
	The function $e^{\lambda\theta} y\in\mathcal{C}$ is called the characteristic function corresponding to the characteristic value $\lambda$.
\end{definition}

Decomposing $X$ by $\{\varphi_{n}^{(i)}\}_{n\geq 0}$,  an equivalent definition of the eigenvalues can be obtained.
That is $\lambda=\lambda(\mu)$ is a characteristic value of \eqref{eqC} if and only if there exists a $n\geq 0$ such that
\begin{equation}\label{eqCE}
\mathrm{det}[~\lambda I+\frac{n^{2}}{l^{2}} D(\mu)-L(\mu)(e^{\lambda \cdot}I)~] = 0.
\end{equation}
Moreover, we call that $\lambda=\lambda(\mu)$ is mode of $n$, and write as $\lambda(mod(n))$.

To consider the Hopf-zero bifurcation, we assume that for some $\mu_{0}=(\mu_{0,1},\mu_{0,2})\in {\mathbb{R}}^{2}$, the following condition is holds:\\
$(\textbf{H}1)$ There exists a neighborhood $O$ of $\mu_{0}$ such that for $\mu\in O$,  \eqref{eqC} has a pair of complex simple conjugate eigenvalues $\alpha(\mu)\pm \mathrm{i}\omega(\mu)(mod(n_1))$ and a simple real eigenvalue $\gamma(\mu)(mod(n_2))$. They all continuously differentiable in $\mu=(\mu_1,\mu_2)$, and satisfy $\alpha(\mu_{0})=0,~\omega(\mu_{0})=\omega_{0}>0,~ \frac{\partial}{\partial \mu_1}\alpha(\mu_{0})\neq 0$, $\gamma(\mu_{0})=0,~\frac{\partial}{\partial \mu_2}\gamma(\mu_{0})\neq 0 ~.$ In addition, all other eigenvalues have non-zero real part.


\subsection{Decomposition of the phase space}
In order to study the qualitative behavior near the Hopf-zero bifurcation, 
we introducing a new phase space  $\mathcal{BC}$ with  $$\mathcal{BC}:=\{\psi:[-1,0]\rightarrow X_{\mathbb{C}} : \psi~\mathrm{is~ continuous~ on}~[-1,0),~\exists \lim_{\theta\rightarrow0^{-}}\psi(\theta)\in X_{\mathbb{C}}\}.$$
Let $\alpha=\mu-\mu_0$ as a function in ${C}:=C([-1,0],\mathbb{C}^{2})$, which satisfies $\frac{d}{dt}{\alpha}(t)=0$. 
Then, \eqref{A} is equivalent to an abstract ordinary differential equation (ODE) on $\mathcal{BC}$:
\begin{equation}\label{eqF}
\frac{d}{dt}U^{t}=AU^{t}+X_{0}F_{0}(\alpha,U^{t}).
\end{equation}
Here $A$ is a operator form $~\mathcal{C}_{0}^{1}=\{\varphi\in\mathcal{C}:\dot{\varphi}\in\mathcal{C},~\varphi(0)\in dom(\Delta)\}$ to $\mathcal{BC}$, defined by 
$$A\varphi=\dot{{\varphi}}+X_{0}[D_0\Delta\varphi(0)+L_{0}(\varphi)-\dot{{\varphi}}(0)],$$
with  $D_0=D(\mu_{0})$, $L_{0}:\mathcal{C}\rightarrow X_{\mathbb{C}}$ is a linear operator given by
$L_0(\varphi)=L(\mu_{0})(\varphi),$ and
$$X_{0}(\theta)=\left\{\begin{array}{cc}
0,&~-1\leq\theta<0,\\
I,&~\theta=0.
\end{array}\right.$$
And $F_0:\mathbb{R}^2\times\mathcal{C}\rightarrow X_{\mathbb{C}}$ is a nonlinear operator and defined by
\begin{equation}\label{F0}
F_{0}(\alpha,\varphi)=[D(\alpha+\mu_{0})-D_0]\Delta \varphi(0)+[L(\alpha+\mu_{0})-L_{0}](\varphi)+F(\alpha+\mu_{0},\varphi).
\end{equation}

We adopt the technology introduced in \cite{Wu1996Theory} for the decomposition of $\mathcal{C}$ and the methods given in \cite{Faria2000Normal} to complete the decomposition of $\mathcal{BC}$.
Let $\{\phi_{1}(\theta)\beta_{n_{1}}, \phi_{2}(\theta)\beta_{n_{2}}\}$ and $\{\psi_{1}(s)\beta_{n_{1}},\psi_{2}(s)\beta_{n_{2}}\}$ are the eigenfunctions of $A$ and its dual $A^*$ relative to $\{\mathrm{i}\omega_{0},0\}$, respectively. Which satisfy $\phi_{1},\phi_{2}\in C$, 
$\psi_{1},\psi_{2}\in C^*:=C([0,1],\mathbb{C}^2)$, and
\begin{equation*}
(\psi_1,\phi_{1})_1=1,\qquad (\psi_1,\bar{\phi}_{1})_1=0,  \qquad (\psi_2,\phi_{2})_2=1.
\end{equation*}
Here $(\cdot,\cdot)_{k}$ are the adjoint bilinear form on $C^{*}\times C$, 
$$(\alpha,\beta)_{k}=\alpha(0)\beta(0)-\int_{-1}^{0}\int_{0}^{\theta}\alpha(\xi-\theta)d\eta_{k}(\theta)\beta(\xi)d\xi,\quad k=1,2, $$
and  $\eta_{k}\in BV([-1,0],\mathbb{C}^{2\times2})$ are given by
$$-\frac{n_{k}^{2}}{l^{2}}D_0\varphi(0)+L_{0}(\varphi)=\int_{-1}^{0} d\eta_{k}(\theta)\varphi(\theta),~~~~~\quad \varphi \in C, \;k=1,2.$$
For simplicity, we denote $\Phi_1(\theta)=(\phi_{1}(\theta),~~\bar{\phi}_1(\theta)),$ $\Psi_1(s)=\left(\psi_{1}(s)^\mathrm{T},\bar{\psi}_1(s)^\mathrm{T}\right)^\mathrm{T}$ and $\Phi_2(\theta)=\phi_{2}(\theta)$,  $\Psi_2(s)=\psi_2(s)$.

Now $\mathcal{BC}$ can be divided into a direct sum of center subspace and its complementary space:
\begin{equation}\label{eqBC}
\mathcal{BC}=\mathcal{P}\oplus \mathrm{ker} \pi,
\end{equation}
where $\pi:\mathcal{BC}\rightarrow\mathcal{P}$ is the projection defined by
\begin{equation}\label{pi}
\pi \varphi=\sum^{2}_{k=1}\Phi_{k}(\Psi_{k},\langle \varphi(\cdot),\beta_{n_{k}}\rangle)_{k}\beta_{n_{k}}.
\end{equation}
Then $U^{t}\in\mathcal{C}_{0}^{1}$ can be divided into 
\begin{equation}\label{U^{t}}
\begin{split}
U^{t}(\theta)=&\sum^{2}_{k=1}\Phi_{k}(\theta)\tilde{z}_{k}(t)\beta_{n_{k}}+y(\theta),\\
\end{split}
\end{equation}
with  $\tilde{z}_{k}(t)=(\Psi_{k},\langle U^{t},\beta_{n_{k}}\rangle)_{k}$ and $ y\in {\mathcal{C}_{0}^{1}}\cap\mathrm{ker}\pi :={\mathcal{Q}^{1}}.$
Since $\pi$ can commute with $A$ in $\mathcal{C}_{0}^{1}$, the system \eqref{eqF} on $\mathcal{BC}$ can be identified as a system on $\mathbb{C}^{3}\times\mathrm{ker}\pi$
\begin{equation}\label{zy}
\begin{aligned}
\dot{z}=&Bz+\Psi(0)
\left(\begin{array}{c}
\langle F_{0}(\alpha,\sum^{2}_{k=1}\Phi_{k}\tilde{z}_{k}\beta_{n_{k}}+y),\beta_{n_{1}}\rangle\\
\langle F_{0}(\alpha,\sum^{2}_{k=1}\Phi_{k}\tilde{z}_{k}\beta_{n_{k}}+y),\beta_{n_{2}}\rangle
\end{array}
\right),\\
\frac{d}{dt}y&=A_{1}y+(I-\pi)X_{0}F_{0}(\alpha,\sum^{2}_{k=1}\Phi_{k}\tilde{z}_{k}\beta_{n_{k}}+y),
\end{aligned}
\end{equation}
where $z=(\tilde{z}_{1},\tilde{z}_{2}):=(z_{1},\bar{z}_{1},z_{2})$, $B=\mathrm{diag}~(i\omega_{0},-i\omega_{0},0)$, $\Psi=\mathrm{diag}~(\Psi_{1},\Psi_{2}),$ and $A_{1}$ is defined by $A_{1}:{\mathcal{Q}^{1}}\subset\mathrm{ker}\pi\rightarrow\mathrm{ker}\pi,~A_{1}\varphi=A\varphi~\mathrm{for}~ \varphi\in\mathcal{Q}^{1}.$

\subsection{Reduction and  the normal forms for Hopf-zero bifurcation}
Consider the formal Taylor expansion 
\begin{equation*}
F_{0}(\alpha,\varphi)=\sum_{j=2}^{\infty} \frac{1}{j!}F_{0}^{(j)}(\alpha,\varphi),\qquad\alpha \in {C},~\varphi\in \mathcal{C},
\end{equation*}
with $F_{0}^{(j)}(\cdot,\cdot)\in L^{j}(\tilde{C}\times\mathcal{C},X_{\mathbb{C}})$ is the $j$th $\mathrm{Fr\acute{e}chet}$ derivation of $F_{0}(\cdot,\cdot).$  Then, \eqref{zy} can be written as
\begin{equation}\label{eq14}
\begin{array}{c}
\begin{aligned}
\dot{z}&=Bz+\sum_{j=2}^{\infty} \frac{1}{j!}f_{j}^{1}(z,y,\alpha),\\
\frac{d}{dt}y&=A _{1}y+\sum_{j=2}^{\infty} \frac{1}{j!}f_{j}^{2}(z,y,\alpha).
\end{aligned}
\end{array}
\end{equation}
Here $f_{j}=(f_{j}^{1},f_{j}^{2}),j\geq2$ are defined by
\begin{equation*}\label{f^i}
\begin{aligned}
&f_{j}^{1}(z,y,\alpha)=\Psi(0)
\left(\begin{array}{c}
\langle F_{0}^{(j)}(z,y,\alpha),\beta_{n_{1}}\rangle\\
\langle F_{0}^{(j)}z,y,\alpha),\beta_{n_{2}}\rangle
\end{array}
\right),\\
&f_{j}^{2}(z,y,\alpha)=(I-\pi)X_{0}F_{0}^{(j)}(z,y,\alpha),
\end{aligned}
\end{equation*}
with $F_0^{(j)}(z,y,\alpha):=F_0^{(j)}(\alpha,\sum^{2}_{k=1}\Phi_{k}\tilde{z}_{k}\beta_{n_{k}}+y)$. 

In fact,  the calculation of the normal forms can be considered as a process to eliminate the non-resonant terms in equation \eqref{eq14}.

First of all, we introduce the following notation: for a normed space $Y$, denote $V_{j}^{m+p}(Y)$ the linear space of homogeneous polynomials of degree $j$ in $m+p$ variables $z=(z_{1},\cdots,z_{m}),~\alpha=(\alpha_{1},\cdots,\alpha_{p})$ with coefficients in $Y$,
$$V_{j}^{m+p}(Y)=\left\{\sum_{|(q,l)|=j}c_{(q,l)}z^{q}\alpha^{l}~:~(q,l)\in N_{0}^{m+p},~c_{(q,l)}\in Y \right\}$$
and the norm $|\sum_{|(q,l)|=j}c_{(q,l)}z^{q}\alpha^{l}|=\sum_{|(q,l)|=j}|c_{(q,l)}|_{Y}.$ Define the operator $M_{j}=(M_{j}^{1},M_{j}^{2})$, $j\geq2$ with
\begin{equation}
\begin{aligned}
&M_{j}^{1}:V_{j}^{3+2}(\mathbb{C}^{3})\rightarrow V_{j}^{3+2}(\mathbb{C}^{3})\\
&(M_{j}^{1}p)(z,\alpha)=D_{z}p(z,\alpha)Bz-Bp(z,\alpha)\\
&M_{j}^{2}:V_{j}^{3+2}(\mathcal{Q}_{1})\subset V_{j}^{3+2}(\mathrm{ker}\pi)\rightarrow V_{j}^{3+2}(\mathrm{ker}\pi)\\
&(M_{j}^{2}h)(z,\alpha)=D_{z}h(z,\alpha)Bz-A_{1}h(z,\alpha)
\end{aligned}
\end{equation}
Since $M_{j}$ restricted on $\mathrm{Ker}(M_{j}^{1})^{c}\times\mathrm{Ker}(M_{j}^{2})^{c}$ to $\mathrm{Im}(M_{j}^{1})\times \mathrm{Im}(M_{j}^{2})$ is bijection,  the corresponding inverse operator is well defined and we denote it by $(M_{j})^{-1}$.

\subsubsection{The normal forms up to the second order }
For the second order terms,  ignore the effects of the nonlinear terms of the perturbation parameters, we choose
$\mathrm{Im}(M_{2}^{1})^{\mathrm{c}}$ spanned by
\begin{equation}\label{ImM_2^1^c}
z_{1}z_2e_{1},z_{1}\alpha_{i}e_{1},{\bar{z}_1}z_2e_{2},{\bar{z}_1}\alpha_{i}e_{2},z_{1}{\bar{z}_1}e_{3},z_{2}^{2}e_{3},z_{2}\alpha_{i}e_{3},\quad  i=1,2,
\end{equation}
with $e_{1}=(1,0,0)^{\mathrm{T}},e_{2}=(0,1,0)^{\mathrm{T}},e_{3}=(0,0,1)^{\mathrm{T}}$.
And $\mathrm{Ker} (M_{2}^{1})^{\mathrm{c}}$ spanned by the rest of elements.

Taking the coordinate transformation
\begin{equation*}
(z,y)\rightarrow({z},y)+\frac{1}{2!}(U_{2}^{1}({z},{\alpha}),U_{2}^{2}({z},{\alpha}))
\end{equation*}
with  $U_{2}=(U_{2}^{1},U_{2}^{2})\in V_{2}^{3+2}(\mathbb{C}^{3})\times V_{2}^{3+2}(\mathcal{Q}_{1})$ and given by
\begin{equation*}
U_{2}(z,\alpha)=(M_{2})^{-1}\mathbf{P}_{\mathrm{Im}(M_{j}^{1})\times \mathrm{Im}(M_{2}^{2})}\comp{{f}_{2}(z,0,\alpha)}
\end{equation*}
Since the non-resonance conditions relative to $\Lambda=\{\pm i\omega_{0},0 \}$ is satisfied, the normal forms of \eqref{A} on the center manifold of the origin near $(\alpha_1,\alpha_2)=0$ has the form
\begin{equation}
\dot{z}=Bz+\frac{1}{2!}g_{2}^{1}(z,0,\alpha)+h.o.t.
\end{equation}
with $g_{2}^{1}(z,0,\alpha)=\mathbf{P}_{\mathrm{Im}(M_{2}^{1})^{\mathrm{c}}}\comp f_{2}^{1}(z,0,\alpha).$
\subsubsection{The normal forms up to the third order}
After the computation of the second order normal forms, the new third order term now has the form $\bar{f}_{3}^{1}(z,0,0)=f_{3}^{1}(z,0,0)+\frac{3}{2}[D_{z}f_{2}^{1}(z,0,0)U_{2}^{1}(z,0)+D_{y}f_{2}^{1}(z,0,0)U_{2}^{2}(z,0)-D_{z}U_{2}^{1}(z,0)g_{2}^{1}(z,0,0)]$. 

We ignore the effects of the perturbation parameters, and choose $\mathrm{Im}(M_{3}^{1})^{\mathrm{c}}$ generated by
\begin{equation}\label{ImM_3^1^c}
\{z_{1}^{2}{\bar{z}_1}e_{1},z_{1}z_{2}^{2}e_{1},z_{1}{\bar{z}_1}^{2}e_{2},{\bar{z}_1}z_{2}^{2}e_{2},z_{1}{\bar{z}_1}z_2e_{3},z_{2}^{3}e_{3}\},
\end{equation}
Then the normal forms up to the third order is given by
\begin{equation}\label{normalform3}
\dot{z}=Bz+\frac{1}{2!}g_{2}^{1}(z,0,\alpha)+\frac{1}{3!}g_{3}^{1}(z,0,0)+h.o.t.
\end{equation}
with $g_{3}^{1}(z,0,0)=\mathbf{P}_{\mathrm{Im}(M_{3}^{1})^{\mathrm{c}}}\comp \bar{f}_{3}^{1}(z,0,0).$
\subsubsection{An equivalence condition of $U_2^2(z,0)\in V_{2}^{3}(\mathcal{Q}_{1})$ }
For the new third order term $\bar{f}_{3}^{1}(z,0,0)$, it requires us to calculate $(1)\;U_{2}^{1}(z,0)\in V_{2}^{3+2}(\mathbb{C}^{3})$, $(2)\;U_{2}^{2}(z,0)\in V_{2}^{3}(\mathcal{Q}_{1})$,
$(3)\;D_{z}f_{2}^{1}(z,0,0):\mathbb{C}^{3}\rightarrow \mathcal{BC}$,
$(4)\;D_{y}f_{2}^{1}(z,0,0):Q_{1}\rightarrow \mathcal{BC}$, first.
Among them, to get $U_{2}^{2}(z,0)$  is always a hard problem.

For convenience, we denote $U_2^2(z,0)$ by $U_2^2(\theta)$ and write it into the following form
\begin{equation*}
U_2^2(\theta)\! =\! h_{200}(\theta)z_1^2\!+\!h_{110}(\theta)z_1{\bar{z}_1}\!+\!h_{101}(\theta)z_1 z_{2}\!+\!h_{020}(\theta)\bar{z}_1^2\!+\!h_{011}(\theta){\bar{z}_1}z_{2}\!+\!h_{002}(\theta)z_{2}^2,
\end{equation*}
with $h_{200},h_{110},h_{101},h_{020},h_{011},h_{002}$
are undetermined. 
In general, $U_2^2(\theta)$ must fulfil three conditions:
\begin{enumerate}[(1)]
	\item  When $\theta \neq 0$,
	\begin{equation}\label{thetaneq0}
	D_z U_2^2(\theta) Bz-D_{\theta} U_2^2(\theta)=-\pi X_0 F_0^{(2)}(z,0,0).
	\end{equation}
	\item When $\theta = 0$, 	
	\begin{equation}\label{theta=0}
	D_z U_2^2(0) Bz-D_0 \Delta U_2^2(0) - L_0(U_2^2)
	= (I-\pi) X_0 F_0^{(2)}(z,0,0).
	\end{equation}
	\item $U_{2}^{2}(\theta)\in V_{2}^{3}(\mathcal{Q}_{1})$.
\end{enumerate} 
We claim that the condition (3) can be replaced by a weaker one.
\begin{theorem}\label{U22} Assume $U_2^2(\theta)$ is one of the solutions of \eqref{thetaneq0} and \eqref{theta=0}. Then  $U_2^2\in V_{2}^{3}(\mathcal{Q}_{1})$ is equivalent to
	\begin{equation}\label{U22=}
	\begin{array}{cc}
	\left( \psi_1, {\langle h_{101}(\cdot),\beta_{n_{1}}\rangle} \right)_1 = 0, &\left( \overline{\psi}_1, {\langle h_{011}(\cdot),\beta_{n_{1}}\rangle} \right)_1 = 0,\\\left( \psi_2, {\langle h_{110}(\cdot),\beta_{n_{2}}\rangle} \right)_2=0,&
	\left( \psi_2, {\langle h_{002}(\cdot),\beta_{n_{2}}\rangle} \right)_2=0.
	\end{array}
	\end{equation}
\end{theorem}
\begin{proof}
	First of all, we claim that $U_2^2\in V_{2}^{3}(\mathcal{Q}_{1})$ means
	\begin{equation*}
	\left( \psi_1, {\langle U_2^2(\cdot),\beta_{n_{1}}\rangle} \right)_1 = 0,\quad
	\left( \overline{\psi}_1, {\langle U_2^2(\cdot),\beta_{n_{1}}\rangle} \right)_1 = 0, \quad
	\left( \psi_2, {\langle U_2^2(\cdot),\beta_{n_{2}}\rangle} \right)_2 = 0.
	\end{equation*}
	
	Taking the inner product to the both sides of \eqref{thetaneq0}-\eqref{theta=0} with $\beta_{n_{1}}$, and using 
	integration by parts, we can obtain that
	\begin{align*}
	&\left( \psi_1, {\langle U_2^2(\cdot),\beta_{n_{1}}\rangle} \right)_1\\
	=&\psi_1(0){\langle U_2^2(0),\beta_{n_{1}}\rangle}-\int_{-r}^{0}\int_{0}^{\theta}\psi_1(0)e^{-\mathrm{i}\omega_0(\xi-\theta)}d\eta_1(\theta){\langle U_2^2(\xi),\beta_{n_{1}}\rangle}d\xi\\
	=&\psi_1(0){\langle U_2^2(0),\beta_{n_{1}}\rangle}\!-\!\frac{1}{\mathrm{i}\omega_0}\int_{-r}^{0}\psi_1(0)d\eta_1(\theta)[{\langle U_2^2(0),\beta_{n_{1}}\rangle}e^{\mathrm{i}\omega_0\theta}\!-\!{\langle U_2^2(\cdot),\beta_{n_{1}}\rangle}]\\
	&~~~-\frac{1}{\mathrm{i}\omega_0}\int_{-r}^{0}\int_{0}^{\theta}\psi_1(\xi-\theta)d\eta_1(\theta)[D_z{\langle U_2^2(\xi),\beta_{n_{1}}\rangle}Bz\!+\!\Phi_1(\xi)\Psi_1(0){\langle F_0^{(2)},\beta_{n_{1}}\rangle}]d\xi\\
	=&\frac{1}{\mathrm{i}\omega_0}\!\left\{\!\psi_1(0)\!\left[\mathrm{i}\omega_0I_d\!+\!\frac{n_1^2}{l^2}D_0\!-\!L_0(e^{\mathrm{i}\omega_0\theta}\cdot I_d)\right]\!{\langle U_2^2(0),\beta_{n_{1}}\rangle}\!+\!D_z\left( \psi_1, {\langle U_2^2(\cdot),\beta_{n_{1}}\rangle}\! \right)_1Bz\right\}\\
	=&\left( \psi_1, {\langle h_{101}(\cdot),\beta_{n_{1}}\rangle} \right)_1.
	\end{align*}
	Which complete 	$\left( \psi_1, {\langle U_2^2(\cdot),\beta_{n_{1}}\rangle} \right)_1 = 0~ 
	\Longleftrightarrow~
	\left( \psi_1, {\langle h_{101}(\cdot),\beta_{n_{1}}\rangle} \right)_1 = 0.$
	Moreover, this also means that under the condition of Theorem \ref{U22}, the establishment of the following equation is natural and does not require any additional conditions.
	\begin{align*}
	&\left( \psi_1, {\langle h_{200}(\cdot),\beta_{n_{1}}\rangle} \right)_1=0,\quad
	\left( \psi_1, {\langle h_{110}(\cdot),\beta_{n_{1}}\rangle} \right)_1=0,\quad
	\left( \psi_1, {\langle h_{020}(\cdot),\beta_{n_{1}}\rangle} \right)_1=0,\\
	&\left( \psi_1, {\langle h_{011}(\cdot),\beta_{n_{1}}\rangle} \right)_1=0,\quad
	\left( \psi_1, {\langle h_{002}(\cdot),\beta_{n_{1}}\rangle} \right)_1=0.
	\end{align*}
	In a similar manner, we can get $\left( \overline{\psi}_1, {\langle U_2^2(\cdot),\beta_{n_{1}}\rangle} \right)_1 \!=\! 0~\Longleftrightarrow~
	\left( \overline{\psi}_1, {\langle h_{011}(\cdot),\beta_{n_{1}}\rangle} \right)_1 \!=\! 0$.
	
	Next, taking the inner product to the both sides of  \eqref{thetaneq0}-\eqref{theta=0} with $\beta_{n_{2}}$, then
	\begin{equation*}
	\begin{aligned}
	&\left( \psi_2, D_z{\langle U_2^2(\cdot),\beta_{n_{2}}\rangle}Bz \right)_2\\
	=&\psi_2(0)\!\left\{\!\!D_z{\langle U_2^2(0),\beta_{n_{2}}\rangle}\!Bz\!\!-\!\!\!\int_{\!-\!r}^{0}\!\!\int_{0}^{\theta}d\eta_2(\theta)[D_{\theta} {\langle U_2^2(\xi),\beta_{n_{2}}\rangle}\!\!-\!\!\Phi_2(\xi)\Psi_{2}(0){\langle F_0^{(2)},\beta_{n_{2}}\rangle}]d\xi\!\!\right\}\\ 
	=&\psi_2(0)\left[-\frac{{n_{2}}^2}{l^2}D_0+L_0(I_d)\right] {\langle U_2^2(0),\beta_{n_{2}}\rangle}\\=&0.
	\end{aligned}
	\end{equation*}
	Which means
	\begin{align*}
	&\left( \psi_2, {\langle h_{200}(\cdot),\beta_{n_{2}}\rangle} \right)_2=0,
	&\left( \psi_2, {\langle h_{101}(\cdot),\beta_{n_{2}}\rangle} \right)_2=0,\\
	&\left( \psi_2, {\langle h_{020}(\cdot),\beta_{n_{2}}\rangle} \right)_2=0,
	&\left( \psi_2, {\langle h_{011}(\cdot),\beta_{n_{2}}\rangle} \right)_2=0.
	\end{align*}
	and 
	\begin{equation*}
	\left( \psi_2, {\langle U_2^2(\cdot),\beta_{n_{2}}\rangle} \right)_2= 0~ 
	\Longleftrightarrow~\left\{
	\begin{aligned}
	&\left( \psi_2, {\langle h_{110}(\cdot),\beta_{n_{2}}\rangle} \right)_2=0\\
	&\left( \psi_2, {\langle h_{002}(\cdot),\beta_{n_{2}}\rangle} \right)_2=0.
	\end{aligned}\right.
	\end{equation*}
	The proof is completed.
\end{proof}
\begin{remark}
	Theorem \ref{U22} shows that, a function which satisfy \eqref{thetaneq0}-\eqref{theta=0}, if it also meet \eqref{U22=}, then it is the function $U_2^2(z,0)\in V_{2}^{3}(\mathcal{Q}_{1})$ that we required.
\end{remark}
%

\section{Normal forms for Turing-Hopf bifurcation}

In this section, we consider a types of Turing-Hopf bifurcation which satisfies $n_1=0$, $n_2\neq 0$ in $(\textbf{H}1)$. In fact, it is the most common situation in application. The explicit formulas of $g_2^1(z,0,\alpha)$ and $g_3^1(z,0,0)$ in the normal forms \eqref{normalform3} are given.

We assume that for some $\mu_0=(\mu_{0,1},\mu_{0,2})\in\mathbb{R}^2$, the following  condition is holds:
$(\textbf{H}2)$ There exists a neighborhood $O$ of $\mu_{0}$ such that for $\mu\in O$,  \eqref{eqC} has a pair of complex simple conjugate eigenvalues $\alpha(\mu)\pm \mathrm{i}\omega(\mu)(mod (0) )$ and a simple real eigenvalue $\gamma(\mu)(mod(n_2\neq0))$. They all continuously differentiable in $\mu=(\mu_1,\mu_2)$, and satisfy $\alpha(\mu_{0})=0,~\omega(\mu_{0})=\omega_{0}>0,~ \frac{\partial}{\partial \mu_1}\alpha(\mu_{0})\neq 0$, $\gamma(\mu_{0})=0,~\frac{\partial}{\partial \mu_2}\gamma(\mu_{0})\neq 0 ~.$ In addition, all other eigenvalues have non-zero real part.

For convenience, we denote
\begin{equation*}
\begin{aligned}
A(\mu) :=& \left[
\begin{array}{cc}
f_{u}(\mu,0,0)&f_{v}(\mu,0,0)\\g_{u}(\mu,0,0)&g_{v}(\mu,0,0)
\end{array}\right],~~~
B(\mu) := \left[
\begin{array}{cc}
f_{u_{\tau}}(\mu,0,0)&f_{v_{\tau}}(\mu,0,0)\\g_{u_{\tau}}(\mu,0,0)&g_{v_{\tau}}(\mu,0,0)
\end{array}\right],
\end{aligned}
\end{equation*}
with $u,v,u_{\tau},v_{\tau}$ are denote as the simplified form of $u(t),v(t),u(t-1),v(t-1)$ in \eqref{A}, respectively. Here $f_{u}(\mu,0,0)=\frac{\partial}{\partial u}f(\mu,0,0)$, and so forth. 
Further assume
\begin{equation*}
\begin{aligned}
F_0^{(2)}(z,y,\alpha)=&\sum\limits_{i=1}^{2}\left[F_{y_i(\theta)z_1}y_i(\theta)z_1\beta_{n_{1}}+F_{y_i(\theta)\bar{z}_1}y_i(\theta)\bar{z}_1\beta_{n_{1}}+F_{y_i(\theta)z_2}y_i(\theta)z_2\beta_{n_{2}}\right]\\& 
+\sum\limits_{i=1}^{2}\left(F_{\alpha_iz_1}\alpha_iz_1\beta_{n_{1}}+F_{\alpha_i\bar{z}_1}\alpha_i\bar{z}_1\beta_{n_{1}}+F_{\alpha_iz_2}\alpha_iz_2\beta_{n_{2}}\right)\\&+\sum\limits_{m+n+k=2}F_{mnk}z_1^m\bar{z}_1^nz_2^k\beta_{n_{1}}^{m+n}\beta_{n_{2}}^k+O(\alpha y,y^2),\\
F_0^{(3)}(z,0,0)=&\sum\limits_{m+n+k=3}F_{mnk}z_1^m\bar{z}_1^nz_2^k\beta_{n_{1}}^{m+n}\beta_{n_{2}}^k.
\end{aligned}
\end{equation*}
Here the formulas for these coefficient vectors $F_{\alpha_iz_j},F_{y_i(\theta)z_j},F_{mnk}$, please refer to Appendix A.
\subsection{Main results of the normal forms} 
Our main results in this section is as follow.
\begin{theorem} \label{Phi-Psi}
	Assume that for system \eqref{A}, the hypothesis (\textbf{H}2) is valid. Then it undergoes a Turing-Hopf bifurcation at $\mu=\mu_0$. Moreover, the center subspace of the phase space $\mathcal{BC}$ and its dual operator $\mathcal{BC}^*$ are  $\mathcal{P}\!=\!span\!\{\phi_1(\theta)\beta_{n_{1}},\bar{\phi}_1(\theta)\beta_{n_{1}},\phi_2(\theta)\beta_{n_{2}}\!\}$ and $\mathcal{P}^*=span\{\psi_1(s)\beta_{n_{1}},\bar{\psi}_1(s)\beta_{n_{1}},\psi_2(s)\beta_{n_{2}}\}$, respectively.
	Here 
	\begin{equation}\label{phi-psi}
	\begin{aligned}
	\phi_1(\theta)&= e^{\mathrm{i}\omega_0\theta}(1\,,\,k_1)^{\mathrm{T}},\,\qquad
	\phi_2(\theta)=(1\,,\,k_3)^{\mathrm{T}},\\
	\psi_1(s)&= e^{\mathrm{-i}\omega_0 s}\,T_1(1\,,\,k_2),
	\quad\psi_2(s)=T_2(1\,,\,k_4),
	\end{aligned}
	\end{equation}
	with 
	\begin{equation*}
	\begin{aligned}
	&k_{1} = \frac{\mathrm{i}\omega_{0}-A_{11}(\mu_0) - B_{11}(\mu_0)e^{\mathrm{-i}\omega_{0}} }{ A_{12}(\mu_0)+B_{12}(\mu_0)e^{\mathrm{-i}\omega_{0}}},\quad
	k_{2} = \frac{\mathrm{i}\omega_{0}-A_{11}(\mu_0) - B_{11}(\mu_0)e^{\mathrm{-i}\omega_{0}} }{A_{21}(\mu_0) + B_{21}(\mu_0)e^{\mathrm{-i}\omega_{0}}},\\
	&	k_{3} = \frac{d_{1} \frac{n_{2}^2}{l^2}-A_{11}(\mu_0) - B_{11}(\mu_0)}{A_{12}(\mu_0)+ B_{12}(\mu_0)},\hspace{1cm}
	k_{4} = \frac{d_{1} \frac{n_{2}^2}{l^2}-A_{11}(\mu_0) - B_{11}(\mu_0)}{A_{21}(\mu_0) + B_{21}(\mu_0)},\\
	&T_{1} = \left\{{k_{1} k_{2}+1 + e^{\mathrm{-i}\omega_{0}} [B_{11}(\mu_0) + B_{21}(\mu_0) k_{2} + k_{1} (B_{12}(\mu_0) + B_{22}(\mu_0) k_{2})]}\right\}^{-1},\\		
	&T_{2} = \left\{[B_{12}(\mu_0) + B_{22}(\mu_0) k_{4} + k_{4}] k_{3} + [B_{11}(\mu_0) + B_{21}(\mu_0) k_{4} + 1]\right\}^{-1}.
	\end{aligned} 	
	\end{equation*}
\end{theorem}
\begin{theorem} \label{th2}
Assume that for system \eqref{A}, the hypothesis (\textbf{H}2) is valid.  Then the three order truncated normal forms near the Turing-Hopf singularity $\mu=\mu_0$ is
 \begin{equation}\label{eqNF3}
	\left\{
	\begin{aligned}
	\dot{z_{1}}=&\hspace{0.48cm}\mathrm{i}\omega_{0}z_{1}\!+\!\frac{1}{2}f_{\alpha_{1}z_1}^{11}\alpha_{1}z_{1}\!+\!\frac{1}{2}f_{\alpha_{2}z_1}^{11}\alpha_{2}z_{1}\!+\!\frac{1}{6}g_{210}^{11} z_1^2{\bar{z}_1} \!+\!\frac{1}{6} g_{102}^{11}z_1z_{2}^2,\\
	\dot{\bar{z}}_1=&-\mathrm{i}\omega_{0}{\bar{z}_1}\!+\!\frac{1}{2}\overline{f_{\alpha_{1}z_1}^{11}}\alpha_{1}{\bar{z}_1}\!+\!\frac{1}{2}\overline{f_{\alpha_{2}z_1}^{11}}\alpha_{2}{\bar{z}_1}\!+\!\frac{1}{6}\overline{g_{210}^{11}}z_1{\bar{z}_1^2}\!+\!\frac{1}{6}\overline{g_{102}^{11}}{\bar{z}_1}z_{2}^2 ,\\
	\dot{z_{2}}=&\hspace{1.67cm}\frac{1}{2}f_{\alpha_{1}z_{2}}^{13}\alpha_{1}z_{2}\!+\!\frac{1}{2}f_{\alpha_{2}z_{2}}^{13}\alpha_{2}z_{2}\!+\!\frac{1}{6}g_{111}^{13}z_1{\bar{z}_1}z_{2}\!+\!\frac{1}{6}g_{003}^{13}z_{2}^3,
	\end{aligned}\right.
	\end{equation}
	with $(\alpha_1,\alpha_2)=(\mu_1-\mu_{0,1},\mu_2-\mu_{0,2})$ and
	\begin{equation*}\label{alphaz}
	\begin{split}
	&f_{\alpha_{1}z_1}^{11} =
	2\psi_1(0)\,\left[\frac{\partial}{\partial \mu_1}A(\mu_0)\phi_1(0)+\frac{\partial}{\partial \mu_1}B(\mu_0)\phi_1(-1)\right],\\
	&f_{\alpha_{2}z_1}^{11}=
	2\psi_1(0)\,\left[\frac{\partial}{\partial \mu_2}A(\mu_0)\phi_1(0)+\frac{\partial}{\partial \mu_2}B(\mu_0)\phi_1(-1)\right],\\
	&f_{\alpha_{1}z_{2}}^{13} =
	2\psi_2(0)\,\left[-\frac{n_2^2}{l^2}\frac{\partial}{\partial \mu_1}D(\mu_0)\phi_2(0)+\frac{\partial}{\partial \mu_1}A(\mu_0)\phi_1(0)+\frac{\partial}{\partial \mu_1}B(\mu_0)\phi_1(-1)\right],\\
	&f_{\alpha_{2}z_{2}}^{13} =
	2\psi_2(0)\,\left[-\frac{n_2^2}{l^2}\frac{\partial}{\partial \mu_2}D(\mu_0)\phi_2(0)+\frac{\partial}{\partial \mu_2}A(\mu_0)\phi_1(0)+\frac{\partial}{\partial \mu_2}B(\mu_0)\phi_1(-1)\right],\\
	\end{split}
	\end{equation*}
	\begin{equation*}\label{g32}
	\begin{split}
	g_{210}^{11} = &f_{210}^{11}+\frac{3}{\mathrm{2i}\omega_{0}}(-f_{110}^{11} f_{200}^{11} +f_{110}^{11} f_{110}^{12} + \frac{2}{3}f_{020}^{11} f_{200}^{12})+\\
	&\frac{3}{2}\psi_1(0)[S_{yz_1}(\langle h_{110}(\theta)\beta_{n_{1}},\beta_{n_{1}}\rangle)  +  S_{y{\bar{z}_1}}(\langle h_{200}(\theta)\beta_{n_{1}},\beta_{n_{1}}\rangle)],\\
	g_{102}^{11} =&f_{102}^{11}+ \frac{3}{\mathrm{2i}\omega_{0}}(-2f_{002}^{11} f_{200}^{11} +f_{002}^{12} f_{110}^{11} +2f_{002}^{11} f_{101}^{13} )+\\
	&\frac{3}{2}\psi_1(0)[S_{yz_1}(\langle h_{002}(\theta)\beta_{n_{1}},\beta_{n_{1}}\rangle)  +  S_{yz_{2}}(\langle h_{101}(\theta)\beta_{n_{2}},\beta_{n_{1}}\rangle)],\hspace{2cm}\\
	g_{111}^{13} =& f_{111}^{13}+\frac{3}{\mathrm{2i}\omega_{0}}(-f_{101}^{13} f_{110}^{11} +f_{011}^{13} f_{110}^{12})+\frac{3}{2}\psi_2(0)[S_{yz_1}(\langle h_{011}(\theta)\beta_{n_{1}},\beta_{n_{2}}\rangle)  +\\&  S_{y{\bar{z}_1}}(\langle h_{101}(\theta)\beta_{n_{1}},\beta_{n_{2}}\rangle)+S_{yz_{2}}(\langle h_{110}(\theta)\beta_{n_{2}},\beta_{n_{2}}\rangle)],\\
	g_{003}^{13} = &f_{003}^{13}+\frac{3}{\mathrm{2i}\omega_{0}}(-f_{002}^{11} f_{101}^{13} +f_{002}^{12} f_{011}^{13})+\frac{3}{2}\psi_2(0)[S_{yz_{2}}(\langle h_{002}(\theta)\beta_{n_{2}},\beta_{n_{2}}\rangle)].
	\end{split}
	\end{equation*}
	Here $f_{mnk}^{12}=\overline{f_{mnk}^{11}}$,
	\begin{equation*}
	\begin{aligned}
	f_{mnk}^{11} =& \frac{1}{\sqrt{l\pi}}\psi_1(0)F_{mnk},\quad
	f_{mnk}^{13} =  \frac{1}{\sqrt{l\pi}}\psi_2(0)F_{mnk},\quad\mathrm{when}\;\; m+n+k=2,\\
	f_{mnk}^{11} =&  \frac{1}{{l\pi}}\psi_1(0) F_{mnk},\qquad
	f_{mnk}^{13} = \frac{1}{{l\pi}}\psi_2(0)F_{mnk},\hspace{0.7cm}\mathrm{when}\;\; m+n+k=3,
	\end{aligned}
	\end{equation*}
	\begin{equation*}\label{h1}
	\begin{aligned}
	&{\langle h_{200}(\theta)\beta_{n_1},\beta_{n_1} \rangle} 
	\!=\!\!\frac{e^{2\mathrm{i}\omega_{0}\theta}}{l\pi}\left[\!2\mathrm{i}\omega_0\!\!-\!\!L_0(e^{\mathrm{2i}\omega_0\cdot}I_d)\!\right]^{\!-\!1}\!F_{200}\!\!-\!\!\frac{1}{\mathrm{i}\omega_{0}\sqrt{l\pi}}[f_{200}^{1,1}\phi_1(\theta)\!\!+\!\!\frac{1}{3}f_{200}^{1,2}\bar{\phi}_1(\theta)],\\
	&{\langle h_{110}(\theta)\beta_{n_1},\beta_{n_1} \rangle} = -\frac{1}{l\pi}[L_0(I_d)]^{-1}F_{110}+\frac{1}{\mathrm{i}\omega_{0}\sqrt{l\pi}}\left[f_{110}^{1,1}\phi_1(\theta)-f_{110}^{1,2}\bar{\phi}_1(\theta)\right],\\
	&{\langle h_{110} (\theta)\beta_{n_2},\beta_{n_2} \rangle}= {\langle h_{110}(\theta)\beta_{n_1},\beta_{n_1} \rangle},\hspace{4cm}\\
	&{\langle h_{101} (\theta)\beta_{n_2},\beta_{n_1} \rangle} =~\frac{e^{\mathrm{i}\omega_0\theta}}{l\pi}\left[\!\mathrm{i} \omega_0+\frac{n_2^2}{l^2}D_0-L_0(e^{\mathrm{i}\omega_0\cdot} I_d)\!\right]^{\!-\!1}\!F_{101}-\frac{1}{\mathrm{i}\omega_{0}\sqrt{l\pi}}f_{101}^{1,3}\phi_2(0),\\
	\end{aligned}
	\end{equation*}
	\begin{equation*}\label{h5}
	\begin{aligned}
	&{\langle h_{011} (\theta)\beta_{n_1},\beta_{n_2} \rangle}=\!\frac{e^{\mathrm{-i}\omega_0\theta}}{l\pi}\left[\!-\mathrm{i} \omega_0\!+\!\frac{n_2^2}{l^2}D_0\!-\!L_0(e^{\mathrm{-i}\omega_0\cdot} I_d)\!\right]^{\!-\!1}\!F_{011}\!+\!\frac{1}{\mathrm{i}\omega_{0}\sqrt{l\pi}}f_{011}^{1,3}\phi_2(0),\\
	&{\langle h_{002} (\theta)\beta_{n_1},\beta_{n_1} \rangle}=-\frac{1}{l\pi}\left[L_0(I_d)\right]^{-1}F_{002}+\frac{1}{\mathrm{i}\omega_{0}\sqrt{l\pi}}\left[f_{002}^{1,1}\phi_1(\theta)-f_{002}^{1,2}\bar{\phi}_1(\theta)\right],\\
	&{\langle h_{002} (\theta)\beta_{n_2},\beta_{n_2} \rangle}=\frac{1}{2l\pi}\left[\frac{(2n_2)^2}{l^2}D_0\!-\!L_0(I_d)\right]^{-1}F_{002}\!+\!{\langle h_{002} (\theta)\beta_{n_1},\beta_{n_1} \rangle},
	\end{aligned}
	\end{equation*}
	and $S_{yz_i} (i=1,2),S_{y\bar{z}_1}$ are linear operators from $ \mathcal{Q}_{1}$ to $X_{\mathbb{C}}$ and given by
	\begin{equation*}\label{S}
	\begin{aligned}
	S_{yz_i}(\varphi) =& (F_{y_1(0)z_i},\;F_{y_2(0)z_i})\varphi(0)+(F_{y_1(-1)z_i},\; F_{y_2(-1)z_i})\varphi(-1),\\
	S_{y\bar{z}_1}(\varphi) =& (\overline{F_{y_1(0){z}_1}},\;\overline{F_{y_2(0){z}_1}})\varphi(0)+(\overline{F_{y_1(-1){z}_1}},\; \overline{F_{y_2(-1){z}_1}})\varphi(-1).\\
	\end{aligned}
	\end{equation*}
	
\end{theorem}
Up to here, we have given the complete algorithm of the normal forms for a general time-delayed reaction-diffusion system \eqref{A} with a Truing-Hopf singularity. The whole calculation process can be automated by Matlab.

\subsection{Proof of Theorem \ref{Phi-Psi}-Theorem \ref{th2}}
In this subsection we will give the proof of Theorem \ref{Phi-Psi}-Theorem \ref{th2}
\begin{proof}[{\bf Proof of Theorem \ref{Phi-Psi}.}]
	Since $\phi_1(\theta),\psi_1(s)$ are satisfy 
	\begin{equation*}
	A\phi_{1}(\theta)\beta_{n_{1}}=\mathrm{i}\omega_0\phi_{1}(\theta)\beta_{n_{1}},\quad A^{*}\psi_{1}(s)\beta_{n_{1}}=\mathrm{i}\omega_0\psi_{1}(s)\beta_{n_{1}},\quad
	\end{equation*} 
	we have $\phi_1(\theta) = e^{\mathrm{i}\omega_0\theta}\phi_1(0)$, $\psi_1(s) = e^{\mathrm{-i}\omega_0s}\psi_1(0)$. Meanwhile, from
	\begin{equation*}
	\begin{split}
	&\left[A(\mu_0)+B(\mu_0)e^{\mathrm{-i}\omega_0}\mathrm{-i}\omega_0\cdot I_d\right]\cdot\phi_1(0)=0,\\
	&\psi_1(0)\cdot\left[A(\mu_0)+B(\mu_0)e^{\mathrm{-i}\omega_0}\mathrm{-i}\omega_0\cdot I_d\right]=0,\\
	& (\psi_{1},\phi_{1})_{1}=1,
	\end{split}
	\end{equation*} 
	we  chose
	\begin{equation*}
	\begin{aligned}
	&k_{1} = \frac{\mathrm{i}\omega_{0}-A_{11}(\mu_0) - B_{11}(\mu_0)e^{\mathrm{-i}\omega_{0}} }{ A_{12}(\mu_0)+B_{12}(\mu_0)e^{\mathrm{-i}\omega_{0}}},\quad
	k_{2} = \frac{\mathrm{i}\omega_{0}-A_{11}(\mu_0) - B_{11}(\mu_0)e^{\mathrm{-i}\omega_{0}} }{A_{21}(\mu_0) + B_{21}(\mu_0)e^{\mathrm{-i}\omega_{0}}},\\
	\end{aligned} 	
	\end{equation*}
	and eventually calculate
	\begin{align*}
	&T_{1} = \left\{{k_{1} k_{2}+1 + e^{\mathrm{-i}\omega_{0}} [B_{11}(\mu_0) + B_{21}(\mu_0) k_{2} + k_{1} (B_{12}(\mu_0) + B_{22}(\mu_0) k_{2})]}\right\}^{-1},	
	\end{align*}
	In the same way, $k_3,k_4,T_2$ can be obtained. This complete the proof.
\end{proof}
In order to complete the proof Theorem \ref{th2} , 
we separate $\bar{f}_{3}^{1}(z,0,0)$ into the following four parts
\begin{equation*}
\begin{aligned}
G_1(z) =& f_{3}^{1}(z,0,0),\hspace{2.2cm}
G_2(z) = \frac{3}{2}D_{z}f_{2}^{1}(z,0,0)U_{2}^{1}(z,0),\\
G_3(z) =& \frac{3}{2}D_{y}f_{2}^{1}(z,0,0)U_{2}^{2}(z,0),\quad
G_4(z) = \frac{3}{2}D_{z}U_{2}^{1}(z,0)g_{2}^{1}(z,0,0).
\end{aligned}
\end{equation*}
And start with the following five lemmas.
\begin{lemma}\label{lemma1}
	Assume that for system \eqref{A}, the hypothesis (\textbf{H}2) is valid. Then the quadratic term in the normal form $\dot{z}=Bz+\frac{1}{2!}g_{2}^{1}(z,0,\alpha)+h.o.t.$ near the Turing-Hopf singularity $\mu=\mu_0$ is 
	\begin{equation*}\label{g21}
	g_2^1(z,0,\alpha) = \left(
	\begin{array}{c}
	f_{\alpha_{1}z_1}^{11}\alpha_{1}z_1+f_{\alpha_{2}z_1}^{11}\alpha_{2}z_1 \\
	\overline{f_{\alpha_{1}z_1}^{11}}\alpha_{1}\bar{z}_1+\overline{f_{\alpha_{2}z_1}^{11}}\alpha_{2}\bar{z}_1 \\
	f_{\alpha_{1}z_{2}}^{13}\alpha_{1}z_{2}+f_{\alpha_{2}z_{2}}^{13}\alpha_{2}z_{2}
	\end{array}\right),
	\end{equation*}
	with $f_{\alpha_{1}z_1}^{11}$, $f_{\alpha_{2}z_1}^{11}$, $f_{\alpha_{1}z_{2}}^{13}$, $f_{\alpha_{2}z_{2}}^{13}$ are given as in Theorem \ref{th2}.
\end{lemma}
\begin{proof}
	From \eqref{F0}, we have the second-order Fr\'{e}chet derivative of $F_{0}(\alpha,\varphi)$ is
	\begin{equation*}\label{F02}
	\begin{split}
	F_{0}^{(2)}(\alpha,\varphi) \!\!=& 2\!\left[\alpha_1\! \frac{\partial}{\partial \mu_1}\!D(\mu_0)\!\!+\!\!\alpha_2\!\frac{\partial}{\partial \mu_2}\!D(\mu_0)\right]\!\!\Delta \varphi(0)\!\!+\!\!2\left[\alpha_1\! \frac{\partial}{\partial \mu_1}\!A(\mu_0)\!\!+\!\!\alpha_2\!\frac{\partial}{\partial \mu_2}\!A(\mu_0)\right]\!\varphi(0)\\&+2\!\left[\alpha_1 \frac{\partial}{\partial \mu_1}B(\mu_0)+\alpha_2\frac{\partial}{\partial \mu_2}B(\mu_0)\right]\varphi(-1)+O(\varphi^2)
	\end{split}
	\end{equation*}
	According to the decomposition \eqref{U^{t}}, we have
	\begin{equation*}
	\begin{aligned}
	&F_{\alpha_{i}z_1} =2\,\left[\frac{\partial}{\partial \mu_i}A(\mu_0)\phi_1(0)+\frac{\partial}{\partial \mu_i}B(\mu_0)\phi_1(-1)\right],\hspace{3.25cm} (i=1,2),\\
	&F_{\alpha_{i}z_{2}} =2\,[-\frac{n_2^2}{l^2}\frac{\partial}{\partial \mu_i}D(\mu_0)\phi_2(0)+\frac{\partial}{\partial \mu_i}A(\mu_0)\phi_2(0)+\frac{\partial}{\partial \mu_i}\textbf{}(\mu_0)\phi_2(-1)],\quad (i=1,2).\\
	\end{aligned}
	\end{equation*}
	Based on the basis of $\mathrm{Im}(M_{2}^{1})^{\mathrm{c}}$, to obtain the final expression of $g_2^1(z,0,\alpha)$, we just need to calculate $f_{\alpha_{1}z_1}^{11},$ $f_{\alpha_{2}z_1}^{11},$ $f_{\alpha_{1}z_{2}}^{13},$ $f_{\alpha_{2}z_{2}}^{13}$, $f_{101}^{11},$ $f_{110}^{13},$ $f_{002}^{13}$.
	
	Due to the fact that 
	$\langle \beta_{n_{1}},\beta_{n_{1}}\rangle = \langle \beta_{n_{2}},\beta_{n_{2}}\rangle=1,$ and
	$$\langle \beta_{n_{1}}\beta_{n_{2}},\beta_{n_{1}}\rangle =
	\langle \beta_{n_{1}}\beta_{n_{1}},\beta_{n_{2}}\rangle=\langle \beta_{n_{2}}\beta_{n_{2}},\beta_{n_{2}}\rangle= 0.$$
	We have $f_{101}^{11}=f_{110}^{13}=f_{002}^{13}=0$, and 
	$f_{\alpha_{1}z_1}^{11}$, $f_{\alpha_{2}z_1}^{11}$, $f_{\alpha_{1}z_{2}}^{13}$, $f_{\alpha_{2}z_{2}}^{13}$ are given as in
	Lemma  \ref{lemma1},
	which complete the proof.
\end{proof}
\begin{lemma} \label{lemmaG1}
	Assume that for system \eqref{A}, the hypothesis (\textbf{H}2) is valid. Then
	\begin{equation} \label{G1}
	\begin{aligned}
	\mathbf{P}_{\mathrm{Im}(M_{3}^{1})^{\mathrm{c}}}\comp G_1(z) =\left( \begin{array}{c}
	f_{210}^{11}z_1^2{\bar{z}_1}+f_{102}^{11}z_1z_{2}^2\\
	\overline{f_{210}^{11}}{\bar{z}_1}^2z_{1}+\overline{f_{102}^{11}}{\bar{z}_1}z_{2}^2\\
	f_{111}^{13}z_1{\bar{z}_1}z_{2}+f_{003}^{13}z_{2}^3
	\end{array}\right)
	\end{aligned}
	\end{equation}	
	with $f_{mnk}^{11} = \frac{1}{{l\pi}}\psi_1(0) F_{mnk}$ and $f_{mnk}^{13} = \frac{1}{{l\pi}}\psi_2(0)F_{mnk}$.
\end{lemma}
\begin{proof}
	The projection \eqref{G1} is obvious, since $\mathrm{Im}(M_{3}^{1})^{\mathrm{c}}$ spanned by 
	\eqref{ImM_3^1^c}.   To complete the proof, we just need to  calculate  $f_{210}^{11}$, $f_{102}^{11}$, $f_{111}^{13}$, $f_{003}^{13}$.
	
	Since
	$\langle \beta_{n_{1}}^3,\beta_{n_{1}}\rangle = \langle \beta_{n_{2}},\beta_{n_{2}}^2\rangle=\langle \beta_{n_{1}}^2\beta_{n_{2}},\beta_{n_{2}}\rangle=\frac{1}{l\pi},\; \langle \beta_{n_{2}}^3,\beta_{n_{2}}\rangle=\frac{3}{2l\pi},$
	and 
	\begin{equation*}
	f_{mnk}^{1}=\Psi(0)
	\left(\begin{array}{c}
	\langle F_{mnk}z_1^m\bar{z}_1^nz_2^k\beta_{n_{1}}^{m+n}\beta_{n_{2}}^k,\beta_{n_{1}}\rangle\\
	\langle F_{mnk}z_1^m\bar{z}_1^nz_2^k\beta_{n_{1}}^{m+n}\beta_{n_{2}}^k,\beta_{n_{2}}\rangle
	\end{array}
	\right)
	\end{equation*}	
	We complete the proof.
\end{proof}

\begin{lemma} \label{lemmaG2}
	Assume that for system \eqref{A}, the hypothesis (\textbf{H}2) is valid. Then
	\begin{equation} \label{G2}
	\begin{aligned}
	\mathbf{P}_{\mathrm{Im}(M_{3}^{1})^{\mathrm{c}}}\comp G_2(z) =\left( \begin{array}{c}
	G_{210}^{21}z_1^2{\bar{z}_1}+G_{102}^{21}z_1z_{2}^2\\
	\overline{G_{210}^{21}}z_1{\bar{z}_1^2}+\overline{G_{102}^{21}}{\bar{z}_1}z_{2}^2\\
	G_{111}^{23}z_1{\bar{z}_1}z_{2}+G_{003}^{23}z_{2}^3
	\end{array}\right)
	\end{aligned}
	\end{equation}	
	with 
	\begin{equation*} 
	\begin{aligned}
	G_{210}^{21}=&\frac{3}{\mathrm{2i}\omega_{0}}(-f_{110}^{11} f_{200}^{11} +f_{110}^{11} f_{110}^{12} + \frac{2}{3}f_{020}^{11} f_{200}^{12}),\\
	G_{102}^{21}=&\frac{3}{\mathrm{2i}\omega_{0}}(-2f_{002}^{11} f_{200}^{11} +f_{002}^{12} f_{110}^{11} +2f_{002}^{11} f_{101}^{13} ),\\
	G_{111}^{23}=&\frac{3}{\mathrm{2i}\omega_{0}}(-f_{101}^{13} f_{110}^{11} +f_{011}^{13} f_{110}^{12}),\\
	G_{003}^{23}=&\frac{3}{\mathrm{2i}\omega_{0}}(-f_{002}^{11} f_{101}^{13} +f_{002}^{12} f_{011}^{13}).\\
	\end{aligned}
	\end{equation*}	
	with $f_{mnk}^{1i}$ $(m+n+k=2,\,i=1,2,3)$ are given as in Theorem \ref{th2}
\end{lemma}
\begin{proof} 
	Since $M_2^1 U_2^{1}(z,0)=\mathbf{P}_{\mathrm{Im} (M_{2}^{1})}\comp{f_2^1(z,0,0)},$ we obtain the unique $U_2^{1}(z,0)\in \mathrm{Ker} (M_{2}^{1})^{\mathrm{c}}$ has the form
	\begin{equation}
	U_2^{1}(z,0) =  \frac{1}{\mathrm{i}\omega_{0}}\left(
	\begin{array}{c}
	f_{200}^{11} z_{1}^2-f_{110}^{11} z_{1} {\bar{z}_1}-\frac{1}{3} f_{020}^{11} {\bar{z}_1}^2
	-f_{002}^{11} z_{2}^2\\   
	\frac{1}{3} f_{200}^{12} z_{1}^2+f_{110}^{12} z_{1} {\bar{z}_1}
	-f_{020}^{12} {\bar{z}_1}^2+f_{002}^{12} z_{2}^2    \\
	f_{101}^{13} z_{1} z_{2}-f_{011}^{13} {\bar{z}_1} z_{2}
	\end{array}\right).
	\end{equation}
	Rewritten $f_2^1(z,0,0) = f_{200}^1z_1^2+f_{110}^1z_1{\bar{z}_1}+f_{101}^1z_1 z_{2}+f_{020}^1{\bar{z}_1}^2+f_{011}^1{\bar{z}_1}z_{2}+f_{002}^1z_{2}^2.$
	Taking the derivative with respect to $z$ and acting on $U_2^{1}(z,0)$, the proof is completed.
\end{proof}

\begin{lemma} \label{lemmaG3}
	Assume that for system \eqref{A}, the hypothesis (\textbf{H}2) is valid. Then
	\begin{equation} \label{G3}
	\begin{aligned}
	\mathbf{P}_{\mathrm{Im}(M_{3}^{1})^{\mathrm{c}}}\comp G_3(z) =\left( \begin{array}{c}
	G_{210}^{31}z_1^2{\bar{z}_1}+G_{102}^{31}z_1z_{2}^2\\
	\overline{G_{210}^{31}}{\bar{z}_1}{\bar{z}_1}z_{2}+\overline{G_{102}^{31}}{\bar{z}_1}z_{2}^2\\
	G_{111}^{33}z_1{\bar{z}_1}z_{2}+G_{003}^{33}z_{2}^3
	\end{array}\right)
	\end{aligned}
	\end{equation}	
	with 
	\begin{equation} \label{G33}
	\begin{aligned}
	G_{210}^{31} =& \frac{3}{2}\psi_1(0)[S_{yz_1}(\langle h_{110}(\theta)\beta_{n_{1}},\beta_{n_{1}}\rangle)  +  S_{y{\bar{z}_1}}(\langle h_{200}(\theta)\beta_{n_{1}},\beta_{n_{1}}\rangle)],\\
	G_{102}^{31} =& \frac{3}{2}\psi_1(0)[S_{yz_1}(\langle h_{002}(\theta)\beta_{n_{1}},\beta_{n_{1}}\rangle)  +  S_{yz_{2}}(\langle h_{101}(\theta)\beta_{n_{2}},\beta_{n_{1}}\rangle)],\\
	G_{111}^{33} =& \frac{3}{2}\psi_2(0)[S_{yz_1}(\langle h_{011}(\theta)\beta_{n_{1}},\beta_{n_{2}}\rangle)  +  S_{y{\bar{z}_1}}(\langle h_{101}(\theta)\beta_{n_{1}},\beta_{n_{2}}\rangle)\\&+S_{yz_{2}}(\langle h_{110}(\theta)\beta_{n_{2}},\beta_{n_{2}}\rangle)],\\
	G_{003}^{33} =&\frac{3}{2}\psi_2(0)[S_{yz_{2}}(\langle h_{002}(\theta)\beta_{n_{2}},\beta_{n_{2}}\rangle)].
	\end{aligned}
	\end{equation}	
	with $S_{yz_i}$ $(i=1,2)$, $S_{y\bar{z}_1}$ and $\langle h_{mnk}(\theta)\beta_{n_{i}},\beta_{n_{j}}\rangle$ $(m+n+k=2,i,j=1,2)$ are given as in Theorem \ref{th2}.
\end{lemma}
\begin{proof}
	First of all, we calculate the Fr\'{e}chet derivative $D_yf_2^1(z,0,0):\mathcal{Q}_{1} \rightarrow X_{\mathbb{C}}$. Since
	\begin{equation*}
	\begin{aligned}
	F_0^{(2)}(z,y,0)=& 
	\sum\limits_{i=1}^{2}\left[F_{y_i(\theta)z_1}y_i(\theta)z_1\beta_{n_{1}}+F_{y_i(\theta)\bar{z}_1}y_i(\theta)\bar{z}_1\beta_{n_{1}}+F_{y_i(\theta)z_2}y_i(\theta)z_2\beta_{n_{2}}\right]\\
	&+\sum\limits_{m+n+k=2}F_{mnk}z_1^m\bar{z}_1^nz_2^k\beta_{n_{1}}^{m+n}\beta_{n_{2}}^k+O(y^2),
	\end{aligned}
	\end{equation*}
	we have $$D_yF_0^{(2)}(z,0,0)(\varphi) = S_{yz_1}(\varphi)z_1\beta_{n_{1}}+S_{y{\bar{z}_1}}(\varphi)\bar{z}_1\beta_{n_{1}}+S_{yz_{2}}(\varphi)z_{2}\beta_{n_{2}},$$
	with $S_{yz_i} \,(i=1,2),S_{y\bar{z}_1}$  are linear operators from $ \mathcal{Q}_{1}$ to $X_{\mathbb{C}}$ and given by
	\begin{equation*}
	\begin{aligned}
	S_{yz_i}(\varphi) =& (F_{y_1(0)z_i},\;F_{y_2(0)z_i})\varphi(0)+(F_{y_1(-1)z_i},\; F_{y_2(-1)z_i})\varphi(-1),\\
	S_{y\bar{z}_1}(\varphi) =& (\overline{F_{y_1(0){z}_1}},\;\overline{F_{y_2(0){z}_1}})\varphi(0)+(\overline{F_{y_1(-1){z}_1}},\; \overline{F_{y_2(-1){z}_1}})\varphi(-1),
	\end{aligned}
	\end{equation*}
	Moreover, we obtain
	\begin{equation*} 
	\begin{aligned}
	&D_yf_2^{1}(z,0,0)(\varphi) = \Psi(0)\left(
	\begin{array}{c}
	\langle D_yF_0^{(2)}(z,0,0)(\varphi),\beta_{n_{1}}\rangle\\
	\langle D_yF_0^{(2)}(z,0,0)(\varphi),\beta_{n_{2}}\rangle
	\end{array}\right).
	\end{aligned}
	\end{equation*}	
	
	Next, we calculate $U_2^2(\cdot)\in V_{2}^{3}(\mathcal{Q}_{1})$. In fact, since $\mathrm{Im}(M_{3}^{1})^{\mathrm{c}}$ spanned by 
	\eqref{ImM_3^1^c}, we just need to calculate
	\begin{equation*}
	\begin{split}
	&{\langle h_{200}(\theta)\beta_{n_1},\beta_{n_1} \rangle},\quad
	{\langle h_{110}(\theta)\beta_{n_1},\beta_{n_1} \rangle},\quad
	{\langle h_{110} (\theta)\beta_{n_2},\beta_{n_2} \rangle},\quad
	{\langle h_{101} (\theta)\beta_{n_2},\beta_{n_1} \rangle},\\
	&{\langle h_{011} (\theta)\beta_{n_1},\beta_{n_2} \rangle},\quad
	{\langle h_{002} (\theta)\beta_{n_1},\beta_{n_1} \rangle},\quad
	{\langle h_{002} (\theta)\beta_{n_2},\beta_{n_2} \rangle}.
	\end{split}
	\end{equation*}
	The process is too tedious to be given here. We just take ${\langle h_{200}(\theta)\beta_{n_1},\beta_{n_1} \rangle}$ as an example, and the same method can be used to others. 
	
	For $h_{200}(\theta)=\sum_{j\geq 0 }\langle h_{200}(\theta),\beta_j \rangle \beta_j,$ it is obvious that
	\begin{equation*}
	\langle h_{200}(\theta)\beta_{n_1},\beta_{n_1} \rangle=\frac{1}{\sqrt{l\pi}}\langle h_{200}(\theta),\beta_{n_1} \rangle.
	\end{equation*}
	Solving $h_{200}(\theta)$ form \eqref{thetaneq0} and \eqref{theta=0}, we have when $\theta\in[-1,0)$,
	\begin{equation*}\label{h200theta}
	h_{200}(\theta)\!=\!e^{2\mathrm{i}\omega_{0}\theta}h_{200}(0)-\frac{(e^{\mathrm{i}\omega_{0}\theta}\!\!-\!e^{2\mathrm{i}\omega_{0}\theta})}{{\mathrm{i}\omega_{0}}}
	f_{200}^{1,1}\beta_{n_1}\!\phi_1(0)-\frac{(e^{\mathrm{-i}\omega_{0}\theta}\!\!-\!e^{2\mathrm{i}\omega_{0}\theta})}{3{\mathrm{i}\omega_{0}}}f_{200}^{1,2}\beta_{n_1}\!\bar{\phi}_1(0).\qquad
	\end{equation*}
	And when $\theta = 0$,
	\begin{equation*}\label{h200}
	(\!-\!2{\mathrm{i} \omega_0}\!+\!D_0\Delta)h_{200}(0)\!+\!L_0(h_{200})=
	f_{200}^{1,1}\beta_{n_1}\!\phi_1(0)+f_{200}^{1,2}\beta_{n_1}\!\bar{\phi}_1(0)+f_{200}^{1,3}\beta_{n_2}\!\phi_2(0)-\!F_{200}\!\beta_{n_1}^2.
	\end{equation*}
	Combine with $[\mathrm{i} \omega_0-L_0(e^{\mathrm{i}\omega_0\cdot} I_d)]{\phi}_1(0) = 0$,
	we can obtain
	\begin{equation*}
	{\langle h_{200}(0),\beta_{n_1} \rangle}= \!-\!\left[\frac{1}{\mathrm{i}\omega_0}f_{200}^{1,1}\phi_1(0)\!\!+\!\!\frac{1}{3\mathrm{i}\omega_0}f_{200}^{1,2}\bar{\phi}_1(0)\right] \!\!-\!\!\frac{1}{\sqrt{l\pi}}\left[-2\mathrm{i} \omega_0\!\!+\!\!L_0(e^{\mathrm{2i}\omega_0\cdot} I_d)\right]^{-1}\!F_{200},
	\end{equation*}
	and eventually get
	\begin{equation*}
	{\langle h_{200}(\theta)\beta_{n_1},\beta_{n_1} \rangle} 
	\!=\!\!\frac{e^{2\mathrm{i}\omega_{0}\theta}}{l\pi}\left[2\mathrm{i} \omega_0\!\!-\!\!L_0(e^{\mathrm{2i}\omega_0\cdot}I_d)\right]^{-1}F_{200}\!-\!\frac{1}{\mathrm{i}\omega_{0}\sqrt{l\pi}}[f_{200}^{1,1}\phi_1(\theta)\!\!+\!\!\frac{1}{3}f_{200}^{1,2}\bar{\phi}_1(\theta)].
	\end{equation*}
	From the proof of Theorem \ref{U22}, it is clear that $ h_{200}(\theta)\in\mathcal{Q}_1.$  In other words, we don't need to consider the condition
	$\left( \psi_1(s), {\langle h_{200}(\theta),\beta_{n_{1}}\rangle} \right)_1=0,$ since it naturally satisfies. Hence the conclusion is proved.
\end{proof}

\begin{lemma}\label{lemmaG4}
	Assume that for system \eqref{A}, the hypothesis (\textbf{H}2) is valid. Then $G_4(z) = 0.$
\end{lemma}
\begin{proof} From Lemma \ref{lemma1}, we have
	$G_4(z) = \frac{3}{2}D_{z}U_{2}^{1}(z,0)g_{2}^{1}(z,0,0)=0,$ 
\end{proof}
\begin{proof}[{\bf Proof of Theorem \ref{th2}}] 
	According to  Lemma \ref{lemmaG1}-\ref{lemmaG4}, we have the cubic terms of the normal form $\dot{z}=Bz+\frac{1}{2!}g_{2}^{1}(z,0,\alpha)+\frac{1}{3!}g_{3}^{1}(z,0,0)+h.o.t.$ near the Turing-Hopf singularity $\mu=\mu_0$ is   
	\begin{equation}\label{g31}
	g_3^1(z,0,0) = \left(
	\begin{array}{c}
	g_{210}^{11} z_1^2{\bar{z}_1} + g_{102}^{11}z_1z_{2}^2 \\
	\overline{g_{210}^{11}}z_1{\bar{z}_1^2}+\overline{g_{102}^{11}}{\bar{z}_1}z_{2}^2 \\
	g_{111}^{13}z_1{\bar{z}_1}z_{2}+g_{003}^{13}z_{2}^3	\end{array}\right)
	\end{equation}
	with $g_{210}^{11}$, $g_{102}^{11}$, $g_{111}^{13}$, $g_{003}^{13}$ are given as in Theorem \ref{th2}.
	Combine with  Lemma \ref{lemma1}, we complete the proof.
\end{proof}

\section{Spatiotemporal patterns}
In this section, we give a further analysis to the normal forms up to the third order. And show the possible spatiotemporal patterns of system \eqref{A} near the Turing-Hopf bifurcation point $\mu=\mu_0$.
First of all, we make a further assumption:\\
(\textbf{H}3) At the Turing-Hopf bifurcation points, in addition to the pair of pure imaginary eigenvalues and the zero eigenvalue,  all other eigenvalues of \eqref{eqC} have strictly negative real parts.

Taking the cylindrical coordinate transformation in \eqref{eqNF3}
\begin{equation*}
\begin{aligned}
z_{1}=\mathcal{R}\cos\Theta+\mathrm{i}\mathcal{R}\sin\Theta,\quad
{\bar{z}_1}=\mathcal{R}\cos\Theta-\mathrm{i}\mathcal{R}\sin\Theta,\quad
z_{2}=\mathcal{V},
\end{aligned}
\end{equation*}
and
re-scaling  the variables with ${r}=\sqrt{\frac{|\mathrm{Re}(g_{210}^{11})|}{6}}\mathcal{R}$, ${v}=\sqrt{\frac{|g_{003}^{13}|}{6}}\mathcal{V}$, $\tilde{t}= t/\varepsilon$, $\varepsilon=\mathrm{Sign}(\mathrm{Re}(g_{210}^{11}))$.
After removing the azimuthal component $\Theta(\cdot)$, a  equivalent planar system is obtained
\begin{equation}\label{rv}
\begin{aligned}
&\frac{dr}{d\tilde{t}}=r[\epsilon_{1}(\mu) + r^{2}+b v^{2}],\\
&\frac{dv}{d\tilde{t}}=v[\epsilon_{2}(\mu) + c r^{2} + d v^{2}].
\end{aligned}
\end{equation}
Here
\begin{equation*}
\begin{aligned} &\epsilon_{1}(\mu)=~\frac{\varepsilon}{2}\,\left[\mathrm{Re}(f_{\alpha_{1}z_1}^{11})(\mu_1-\mu_{0,1})+\mathrm{Re}(f_{\alpha_{2}z_1}^{11})(\mu_2-\mu_{0,2})\right],\\
&\epsilon_{2}(\mu)=~\frac{\varepsilon}{2}\,\left[f_{\alpha_{1}z_{2}}^{13}(\mu_1-\mu_{0,1})+f_{\alpha_{2}z_{2}}^{13}(\mu_2-\mu_{0,2})\right], \\ 
&b =~ \frac{\varepsilon\,\mathrm{Re}(g_{102}^{11})}{|g_{003}^{13}|} ,\quad
c =~ \frac{\varepsilon \,g_{111}^{13}}{|\mathrm{Re}(g_{210}^{11})|} ,\quad
d =~ \frac{\varepsilon \,g_{003}^{13}}{|g_{003}^{13}|}=\pm1.  
\end{aligned}
\end{equation*}
For system \eqref{rv}, there are at most four positive equilibrium points.
\begin{equation}
\begin{aligned}
&E_1 = (0,0),  \hspace{4cm} \mathrm{for~all}\;\; \epsilon_1, \epsilon_2,\\
&E_2 = (\sqrt{-\epsilon_1},0),\hspace{3.3cm} \mathrm{for}\; \epsilon_1<0,\\
&E_3 = (0,\sqrt{-\frac{\epsilon_2}{d}}),\hspace{3.2cm} \mathrm{for}\;\; \frac{\epsilon_2}{d}<0,\\
&E_4 = (\sqrt{\frac{b\epsilon_2-d\epsilon_1}{d-bc}},\sqrt{\frac{c\epsilon_1-\epsilon_2}{d-bc}}),\hspace{0.95cm}\mathrm{for}\;\frac{b\epsilon_2-d\epsilon_1}{d-bc},\frac{c\epsilon_1-\epsilon_2}{d-bc}>0.
\end{aligned}
\end{equation}
Depending on the different signs of $b,c,d$ and $d-bc$, there are 12 distinct types of unfolding for \eqref{rv}, the associated phase portraits could be found in \cite{Phillp1988,Wang2010Hopf}. 

We claim that the solution of the original system \eqref{A} restrict on center manifold is homeomorphic to 
\begin{equation}\label{R}
\begin{split}
W^{t}(\theta)
=&\mathcal{R}(t)\,[\phi_{1}(\theta)e^{\mathrm{i}\Theta(t)}+\bar{\phi}_{1}(\theta)e^{\mathrm{-i}\Theta(t)}]\beta_{n_1}+\phi_2(\theta)\mathcal{V}(t)\beta_{n_2}.
\end{split}
\end{equation}
On the basis of the center manifold emergence theorem given by \cite[chap.6]{carr2012applications} and \cite[chap.5]{Wu1996Theory}, the solutions of the original system \eqref{A}, which the initial functions in a neighborhood $V$ of $0$ in $\mathcal{C}$, are exponentially convergent to the homeomorphism of the attractors of the solutions restrict on center manifold \eqref{R}. That is
\begin{equation}\label{W}
U(t) \approx \mathcal{H}( W(t))+O(e^{-\gamma t}), ~~~~~~~~~\textit{as}~~ t \rightarrow \infty,
\end{equation} 
with $\mathcal{H}(\cdot)$ is an isomorphic mapping and $\gamma>0$.

Totally speaking, there are five possible attractors for \eqref{rv} which  corresponding to five different types of attractor for the original system \eqref{A}. The connection has been shown in Table \ref{tab1}.
\begin{table}[htbp]
	\caption{The correspondence between the planner and original system}
	\label{tab1}
	\centering
	{\begin{tabular}{|c|c|}
			\hline
			\bf{Planar system \eqref{rv}} & \bf{The original system (\ref{A}})\\
			\hline
			$E_1$         & Constant steady state $(0,0)$\\
			$E_2$         & Spatially homogeneous periodic solution\\
			$E_3$         & Non-constant steady state\\
			$E_4$         & Spatially non-homogeneous periodic solution\\
			Periodic solution & Spatially non-homogeneous quasi-periodic solution\\
			\hline
	\end{tabular}}
\end{table} 
In the following, we will give a detailed analysis to each attractor. 
In each case, we assume that the parameters $\mu\in\mathbb{R}^2$ are located in the appropriate area $S_i~(i=1,2,3,4,5)$, which near the bifurcation point $\mu_0$ and ensures the appearance of this situation. 
Without loss of generality, further assume $\varepsilon=1$.  
\vskip 0.2cm
\noindent\textbf{Case 1}. $E_1$ is a stable attractor of system $\eqref{rv}$, when $\mu\in S_1\subset\mathbb{R}^2$.

In this case, we assert that the zero equilibrium of the original system \eqref{A} is also a stable attractor for $\mu\in S_1$. In fact, when the initial value $(r_0,v_0)$ of \eqref{rv} is sufficiently close to the 0, we have
\begin{equation*}
(r,v)\rightarrow E_1, \hspace{1cm}\textit{as}~~ t\rightarrow \infty.
\end{equation*}
From \eqref{W}, there exists a neighborhood $Q_1$ of 0 in $\mathcal{C}$,  such that when the initial functions $(u_0,v_0)\in Q_1$, the corresponding solution
\begin{equation*}
U(t)\rightarrow (0,0),\hspace{1cm}\textit{as}~~t\rightarrow \infty.
\end{equation*}
Which means zero is a attractor of the original system \eqref{A}.
\vskip 0.2cm
\noindent\textbf{Case 2}. $E_2$ is a stable attractor of system $\eqref{rv}$, when  $\mu\in S_2\subset\mathbb{R}^2$.

The same as above, for system \eqref{A},  there exists a region $Q_2\in V$,  such that the solutions
\begin{equation*}
U(t)\rightarrow \mathcal{H}\left(\rho_2\left[\phi_{1}(0)e^{\mathrm{i}\omega_2 t}+\bar{\phi}_{1}(0)e^{\mathrm{-i}\omega_2 t}\right]\right), \hspace{1cm}\textit{as}~~t\rightarrow \infty,
\end{equation*}
when the initial functions $(u_0,v_0)\in Q_2$. Here $\rho_2 \sim \sqrt{{-6\epsilon_1 }/{|\mathrm{Re}(g_{210}^{11})|}}\beta_{n_{1}}$, $\omega_2\sim \omega_0+\frac{1}{2}[\mathrm{Im}(f_{\alpha_{1}z_1}^{11})(\mu_1-\mu_{0,1})+\mathrm{Im}(f_{\alpha_{2}z_1}^{11})(\mu_2-\mu_{0,2})]-{\epsilon_1\,\mathrm{Im}(g_{210}^{11}) }/{|\mathrm{Re}(g_{210}^{11})|}.$
That means, the original systems \eqref{A}  exists a stable spatially homogeneous periodic solution which is bifurcate from the zero equilibrium when $\mu\in S_2.$ 
\vskip 0.2cm
\noindent\textbf{Case 3}. $E_3$ is a stable attractor of system $\eqref{rv}$, when $\mu\in S_3\subset\mathbb{R}^2$. (see e.g. {\uppercase\expandafter{\romannumeral 4}}b,
{\uppercase\expandafter{\romannumeral 6}}a, {\uppercase\expandafter{\romannumeral 7}}a,  {\uppercase\expandafter{\romannumeral 7}}b in \cite[chap.7]{Phillp1988}).
%

In this case, there exist a region $Q_3\in V$, such that the solutions of  \eqref{A},
\begin{equation*}
U(t)\rightarrow \mathcal{H}\left(h_3\cos(\frac{n_2}{l}x)\right), \hspace{1cm}\textit{as}~~t\rightarrow \infty,
\end{equation*}
when the initial functions $(u_0,v_0)\in Q_3$. Here $h_3 \sim \sqrt{\frac{-6\epsilon_2 }{d|g_{003}^{13}|}}\phi_2(0)$.
Due to the mirror symmetry of the system \eqref{rv}, there are two coexisting spatially inhomogeneous steady-state attractors for the original system \eqref{A}, which are stable and bifurcate from the zero equilibrium when $\mu\in S_3.$ 
\vskip 0.2cm
\noindent\textbf{Case 4}. $E_4$ is a stable attractor of system $\eqref{rv}$, when $\mu\in S_4\subset\mathbb{R}^2$. (see e.g. {\uppercase\expandafter{\romannumeral 6}}a, {\uppercase\expandafter{\romannumeral 7}}a in \cite[chap.7]{Phillp1988}).

In the same way, there exist a region $Q_4\in V$, such that for the initial functions $(u_0,v_0)\in Q_4$, the solutions of  \eqref{A}
\begin{equation*}\label{timespace1}
U(t)\rightarrow \mathcal{H}\left(\rho_4\left[\phi_{1}(0)e^{\mathrm{i}\omega_4 t}+\bar{\phi}_{1}(0)e^{\mathrm{-i}\omega_4 t}\right]+h_4\cos(\frac{n_2}{l}x)\right), \quad\textit{as}~~t\rightarrow \infty,
\end{equation*}
with the parameters $\rho_4 \sim \sqrt{\frac{6(b\epsilon_2-d\epsilon_1) }{(d-bc)|\mathrm{Re}(g_{210}^{11})|}}\beta_{n_{1}}$,  $h_4 \sim \sqrt{\frac{6(c\epsilon_1-\epsilon_2) }{(d-bc)|g_{003}^{13}|}}\phi_2(0)$,   
$\omega_4\sim\omega_0+\frac{1}{2}[\mathrm{Im}(f_{\alpha_{1}z_1}^{11})(\mu_1-\mu_{0,1})+\mathrm{Im}(f_{\alpha_{2}z_1}^{11})(\mu_2-\mu_{0,2})]+\frac{\mathrm{Im}(g_{210}^{11})(b\epsilon_2-d\epsilon_1) }{(d-bc)|\mathrm{Re}(g_{210}^{11})|}+\frac{\mathrm{Im}(g_{102}^{11})(c\epsilon_1-\epsilon_2)}{(d-bc)|g_{003}^{13}|}.$
Due to mirror symmetry of the system \eqref{rv}, the original systems \eqref{A}  possesses two  coexisting spatially inhomogeneous periodic solutions, which are stable and bifurcate from the zero equilibrium when $\mu\in S_4.$ These solutions could be used to explain the existence of the spatiotemporal
patterns which has been proposed in \cite{Maini1997Spatial}.
The spatiotemporal 
patterns,  which have defined spatial wavelengths (waveforms) and oscillate periodically with time,
have been shown in Figure \ref{fig1}.

\begin{figure}[htbp]
	\centering                           
	\subfigure[$H_1(t,x)$]{       
		\begin{minipage}{0.48\linewidth} 
			\centering                                      
			\includegraphics[scale=0.2]{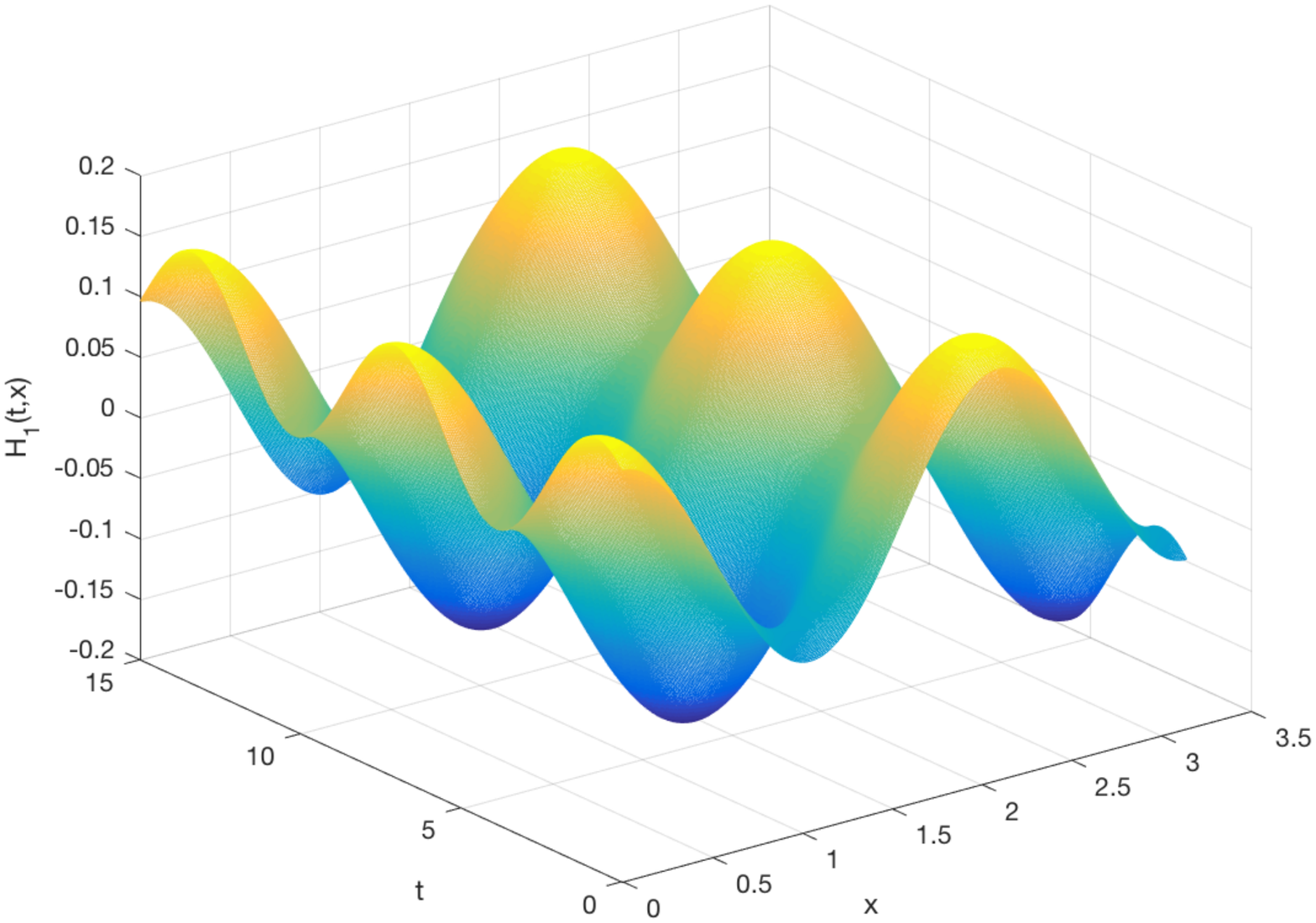}            
	\end{minipage}}
	\subfigure[$H_2(t,x)$]{                  
		\begin{minipage}{0.48\linewidth} 
			\centering                                     
			\includegraphics[scale=0.2]{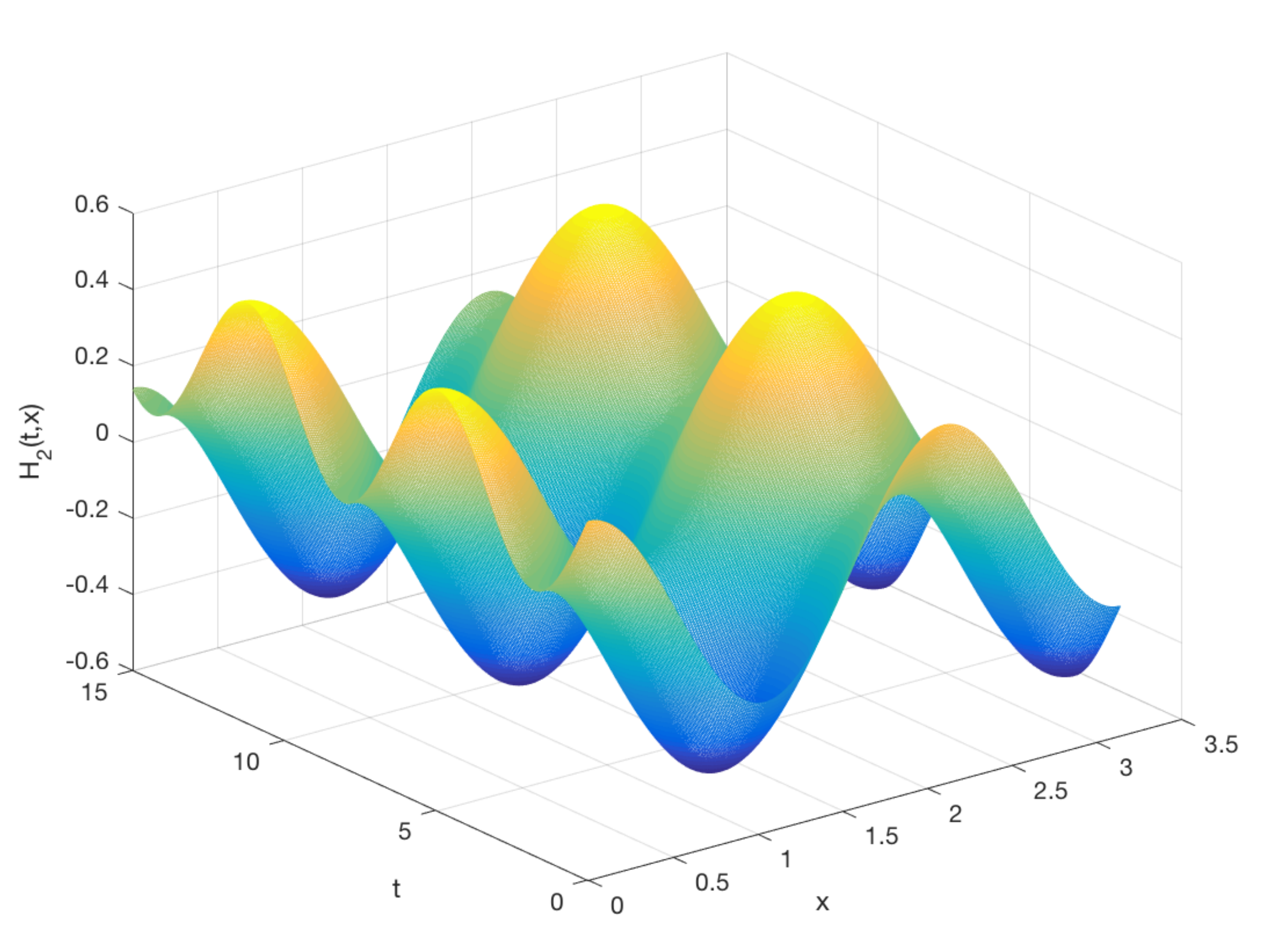}                
		\end{minipage}
	}\caption{The spatially non-homogeneous periodic attractors $H=(H_1,H_2)$, under the {Case 4}. Here $\rho_4 = 0.2,~\phi_1(0) = ( 0.1-0.1\mathrm{i},0.1+0.5\mathrm{i})^{\mathrm{T}}$, $h_4 = (0.1,0.3)^{\mathrm{T}}$, $w_4  =1$, $l=1$, $n_2 = 3$.} 
	\label{fig1}                                                    
\end{figure}

\noindent\textbf{Case 5}. 
There exists a area $S_5\subset\mathbb{R}^2$, such that $E_4$ of  \eqref{rv} loses its stability due to the occurrence of Hopf bifurcation. At the same time, a periodic solution is generated (see e.g. {\uppercase\expandafter{\romannumeral 6}}a, {\uppercase\expandafter{\romannumeral 7}}a in \cite[chap.7]{Phillp1988}).

This is the most complicated but also the most interesting one. The bifurcation analysis in \cite{Phillp1988} has shown,  when $d = -1$, $b>0$, $c<0$ and $d-bc>0$, a Hopf bifurcation can occur. The corresponding result is $E_4$ becomes unstable, and a periodic solution $E_4+(\rho\cos(\varpi t),\rho\sin(\varpi t))$ arises . Omit the tedious derivation, there is a region $Q_5\in V$, such that when the initial functions in $Q_5$, the solutions of $\eqref{A}$ 
\begin{equation*}\label{timespace2}
U(t)\rightarrow \mathcal{H}\left(\![\rho_4\!+\!\rho_5\cos(\varpi t)][\phi_{1}(0)e^{\mathrm{i}\omega_5 t}\!+\!\bar{\phi}_{1}(0)e^{\mathrm{-i}\omega_5 t}]\!+\![h_4\!+\!h_5\sin(\varpi t)]\cos(\frac{n_2}{l}x)\!\right),
\end{equation*}
as $t\rightarrow \infty$. Here $\rho_4,\rho_5$ are given as in case 4, $\rho_5 \sim \rho\sqrt{{6}/{|\mathrm{Re}(g_{210}^{11})|}}\beta_{n_{1}}$,  
$h_5 \sim \rho_4\sqrt{{6}/{|g_{003}^{13}|}}\phi_2(0)$, 
$\varpi\sim 2\sqrt{(b\epsilon_2-d\epsilon_1)(c\epsilon_1-\epsilon_2)/(d-bc)}$, 
$\omega_5\sim\omega_4$, $\rho$ is sufficiently small.
Speaking generally, the attractor of the original system \eqref{A} is a superposition of multiple periodic functions. Due to \eqref{W} and mirror symmetry, there are  two  coexisting spatially inhomogeneous quasi-periodic solutions for \eqref{A}, when $\mu\in S_5$.  which are stable and obtained by the secondary Hopf bifurcation. The patterns of such solutions are expected to be similar as the one demonstrated in Figure \ref{fig2}.  

{{These spatially non-homogeneous quasi-periodic solutions have the following features: 
		
		\textbf{(1)} The temporal pattern (for any fixed $x=x_0\in\Omega$) is quasi-periodic about time. Moreover, these time-quasi-periodic waves also change periodically about space variable $x$ (\textit{i.e.,}  $U(\cdot,x_0)$ is quasi-period about time, and there is a constant $T_1$ such that  $U(\cdot,x_0)=U(\cdot,x_0+T_1)$ for any $x_0\in\Omega$).
		
		\textbf{(2)} The spatial patterns (for any fixed $t=t_0>0$) change quasi-periodically (include amplitude, spatial waveform) with time (\textit{i.e.,} there is a constant $T_2$ such that  $U(t_0,\cdot)=U(t_0+T_2,\cdot)$ for any $t_0>0$).} }

\begin{figure}[htbp]
	\centering                           
	\subfigure[$H_1(t,x)$]{       
		\begin{minipage}{0.48\linewidth} 
			\centering                                      
			\includegraphics[scale=0.2]{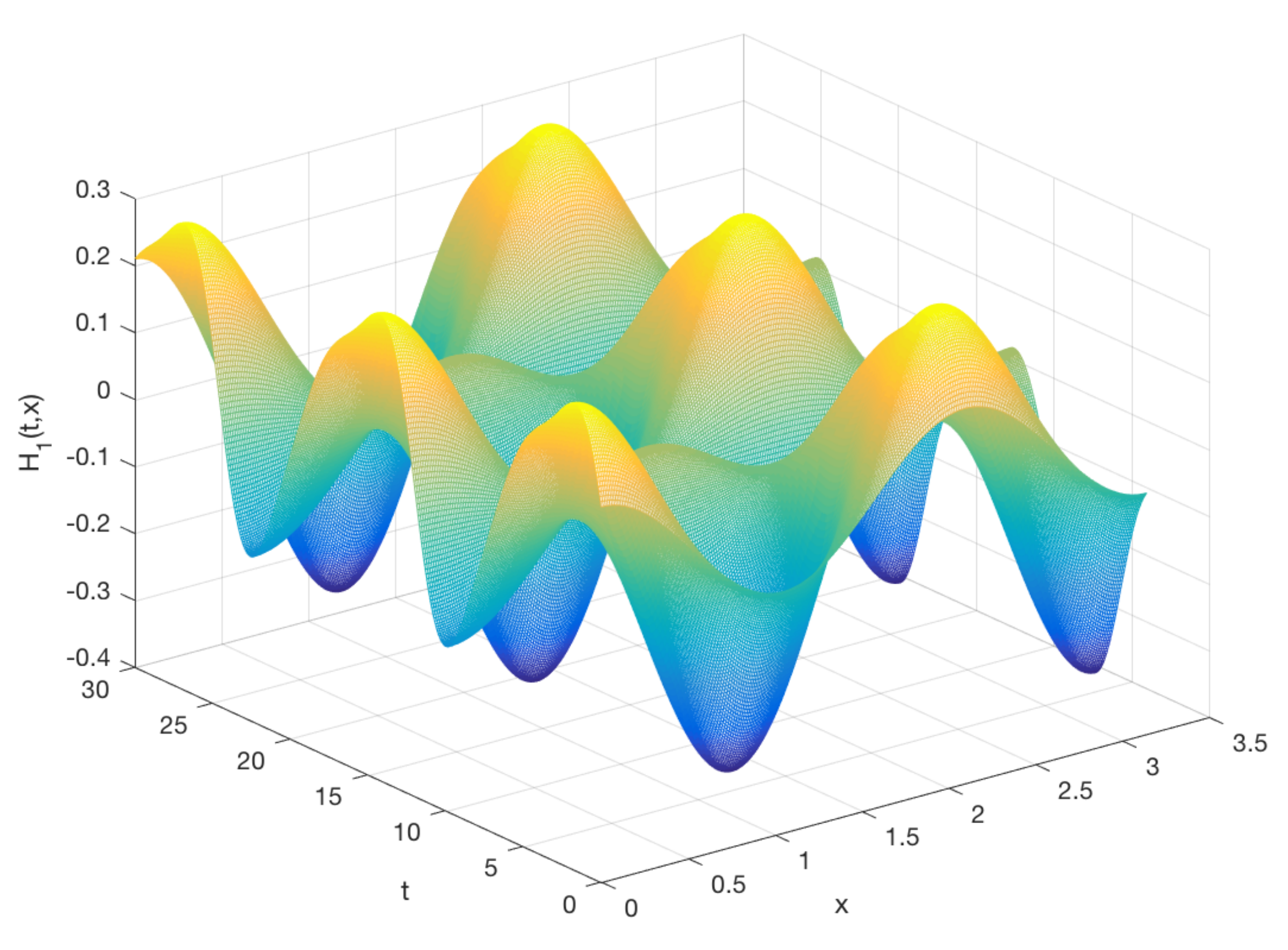}            
	\end{minipage}}
	\subfigure[$H_2(t,x)$]{                  
		\begin{minipage}{0.48\linewidth} 
			\centering                                     
			\includegraphics[scale=0.2]{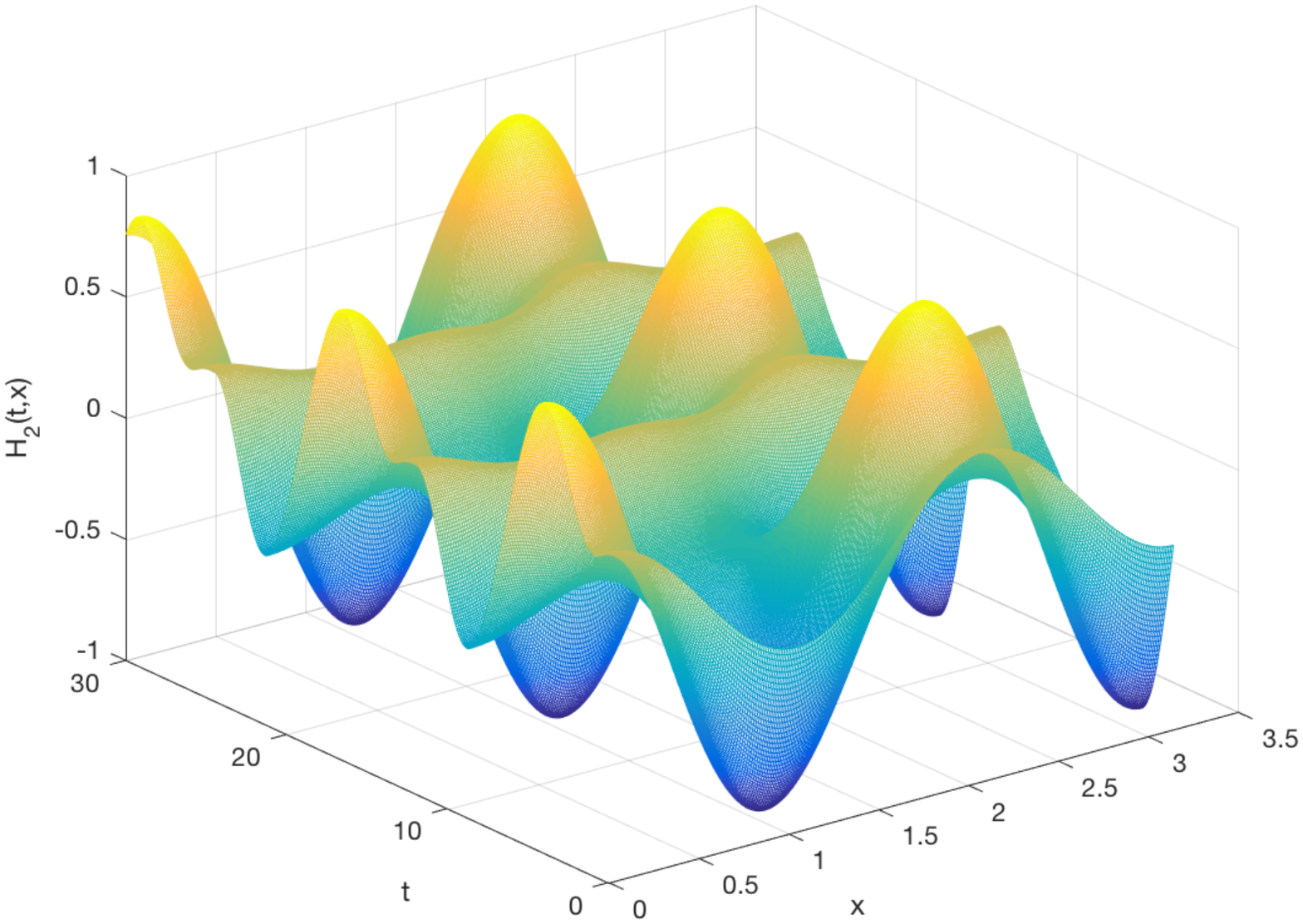}                
		\end{minipage}
	}
	\caption{The spatially non-homogeneous quasi-periodic attractors $H=(H_1,H_2)$, under the case 5. Here $\rho_4 = 0.2, \rho_5 = 0.1, \phi_1(0) = ( 0.1-0.1\mathrm{i},0.1+0.5\mathrm{i})^{\mathrm{T}}$, $h_4 = (0.1,0.3)^{\mathrm{T}}$, $h_5 = (0.2,0.5)^{\mathrm{T}}$, $w_5  =1$, $\varpi = 0.5$, $l=1$, $n_2 = 3$.	  	} 
	\label{fig2}                                                    
\end{figure}

\begin{remark}
	The above analysis shows that, the complex spatiotemporal inhomogeneous solutions can be bifurcated from the zero equilibrium, since the existence of Turing-Hopf bifurcation.  These spatially inhomogeneous periodic and quasi-periodic solutions, which mentioned in \textbf{Case 4} and \textbf{Case 5}, can't be explained by a simple Hopf bifurcation or a simple steady-state bifurcation. Turing-Hopf bifurcation provides a theoretical basis for the formation of spatiotemporal patterns.
\end{remark}

\section{Example and numerical simulation}

In this section, we do a brief study of the Turing-Hopf bifurcation in Holling-Tanner model with a discrete time delay (see, e.g. \cite{Kooij1997Limit,Saez1999Dynamics}),
\begin{equation}\label{eq5.1}
\left\{
\begin{aligned}
&\frac{\mathrm{d}u(t)}{\mathrm{d}t}\!-\!d_{1}\Delta u(t)\!=\!u(t)[1\!-\!u(t)]\!-\!\frac{au(t)v(t)}{u(t)+b},&\!&x\in(0,l\pi),t>0,&\\
&\frac{\mathrm{d}v(t)}{\mathrm{d}t}\!-\!d_{2}\Delta v(t)=rv(t)[1\!-\!\frac{v(t-\tau)}{u(t-\tau)}],&\!&x\in(0,l\pi),t>0,&\\
&\partial_{\eta}u\!=\!\partial_{\eta}v=0,&\!&x=0~\mathrm{or}~l\pi,t>0,&\\
&u(x,t)=\varphi(x,t), v(x,t)=\psi(x,t),&\!&x\in(0,l\pi),t\in[-\tau,0].&\\
\end{aligned}\right.
\end{equation}
In the following, we take $d_1 = 0.1,$ $d_2 = 10.0,$ $a = 1.0,$ $b = 0.1$, $l = 5.0$. And consider the effect of  $\tau$ and $r$ to the system \eqref{eq5.1}. The 3-orders truncated normal forms and some numerical simulations about the spatiotemporal patterns and the spatial patterns are given.

There are only one positive constant steady state $E_{1}=(0.270,0.270)$ in this system \eqref{eq5.1}. The linearied equation at $E_1$ is 
\begin{equation}\label{eq5.2}
\frac{\mathrm{d}}{\mathrm{d}t}U(t)=
{D}\Delta U(t)+{L}(r,\tau)(U_{t}).
\end{equation}
here ${D}=\mathrm{diag}~( d_{1}, d_{2})$, and ${L}(\tau,r)(\cdot):\mathcal{C}\rightarrow X_{\mathbb{C}} $ is a bounded linear operator given by
\begin{displaymath}
{L}(\tau,r)(\phi)={\left(\begin{array}{cc}
	0.2625 &-0.7298\\
	0  &0
	\end{array}\right)} \phi(0)+{\left(\begin{array}{cc}
	0&0\\
	r&-r
	\end{array}\right)} \phi(-\tau).
\end{displaymath}
The corresponding characteristic equations are
\begin{displaymath}\label{characteristic}
\begin{aligned}
&\lambda^{2}-{T_{n}}(\tau,r)\lambda+{D_{n}}(\tau,r)=0,&&n=0,1,2,\cdots&
\end{aligned}\eqno({5.3}_{n})
\end{displaymath}
with
\begin{equation*}
\begin{aligned}
&{T_{n}}(\tau,r)=0.2625-\dfrac{10.1n^{2}}{l^{2}}-re^{-\lambda\tau},&\\
&{D_{n}}(\tau,r)=\dfrac{n^{2}}{l^{2}}(\dfrac{n^{2}}{l^{2}}-2.625)+re^{-\lambda\tau}(\dfrac{0.1n^{2}}{l^{2}}+ 0.4673).&
\end{aligned}
\end{equation*}

\begin{proposition} Taking $d_1 = 0.1,$ $d_2 = 10.0,$ $a = 1.0,$ $b = 0.1$, $l=5.0$ in the Holiing-Tanner model with a discrete delay \eqref{eq5.1}. Then we have,
	\begin{enumerate}[(\romannumeral 1)]
		\item When $(\tau,r)=  (0.4567, 2.8646): = (\tau_0,r_0)$, the lineared equation \eqref{eq5.2} has a pair of conjugate pure imaginary eigenvalues $\pm 2.8899\,\mathrm{i}(mod(0))$ and a $0(mod(5))$ eigenvalue, all other eigenvalues have strictly negative real parts.
		\item At $(\tau_0,r_0) = (0.4567, 2.8646)$, the system \eqref{eq5.1} undergoes a Turing-Hopf bifurcation. 
	\end{enumerate}
\end{proposition}

After taking the time scale transformation $t\to t/\tau$. 
we can get the coefficients in the truncated normal forms \eqref{eqNF3} by the formula given in Theorem \ref{th2},
\begin{equation*}
\begin{aligned}
&f_{\alpha_1z_1}^{11} = 3.5526 + 2.0355\mathrm{i}, \hspace{0.93cm} f_{\alpha_2z_1}^{11} = 0.4523 + 0.42914\mathrm{i}, \\
&f_{\alpha_1z_2}^{13}= 0,\hspace{3.5cm}  f_{\alpha_2z_2}^{13} = -0.0433,\\
&g_{210}^{11} = -25.8208-41.9428\mathrm{i},\hspace{0.55cm} g_{102}^{11} = -0.8398+0.1637\mathrm{i},\\
& g_{111}^{3} = - 0.2315, \hspace{2.6cm}g_{003}^{13} = -0.6018,
\end{aligned}
\end{equation*}
The corresponding amplitude system is 
\begin{equation}
\begin{aligned}
&\frac{d\rho}{d{t}}=\rho[\epsilon_{1}(\tau,r) + \rho^{2}+b_0 v^{2}],\\
&\frac{dv}{d{t}}=v[\epsilon_{2}(\tau,r) + c_0 \rho^{2} + d_0 v^{2}].
\end{aligned}
\end{equation}
with $ b_0 =1.3954, \;c_0 =0.0090,\;d_0 =1,\; \varepsilon=-1,$ and
\begin{equation*}
\begin{aligned}
&\epsilon_{1}(\tau,r)=-1.7763\,(\tau-\tau_0)-0.2261\,(r-r_0),\\
&\epsilon_{2}(\tau,r)=0.0217\,(r-r_0).\\ 
\end{aligned}
\end{equation*}
Since $d_0-b_0c_0=0.9874>0$, we assert that the case  ${\uppercase\expandafter{\romannumeral 1}}_a$ in \cite[chap.7]{Phillp1988}
occurs. The bifurcation sets with parameters $(\alpha_1,\alpha_2)=(\tau-\tau_0,r-r_0)$ near $(0,0)$ is given in Figure \ref{fig3}. 
\begin{figure}[htb]
	\centering 
	\subfigure[]{       
	\begin{minipage}{0.4\linewidth} 
		\centering                                      
		\includegraphics[scale=0.25]{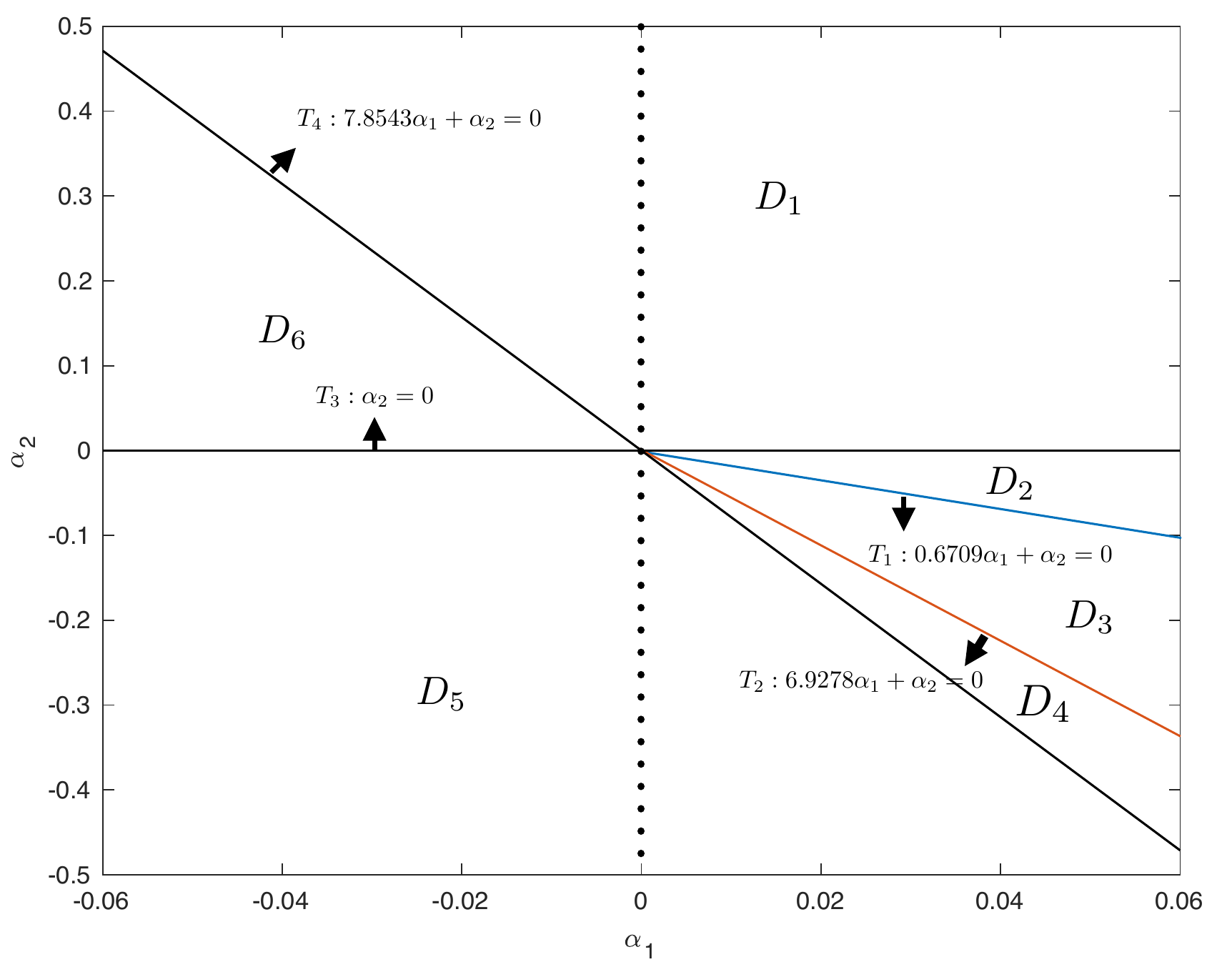} \label{fig3}              
	\end{minipage}}   
	\subfigure[]{                  
	\begin{minipage}{0.55\linewidth} 
		\centering                                     
		\includegraphics[scale=0.25]{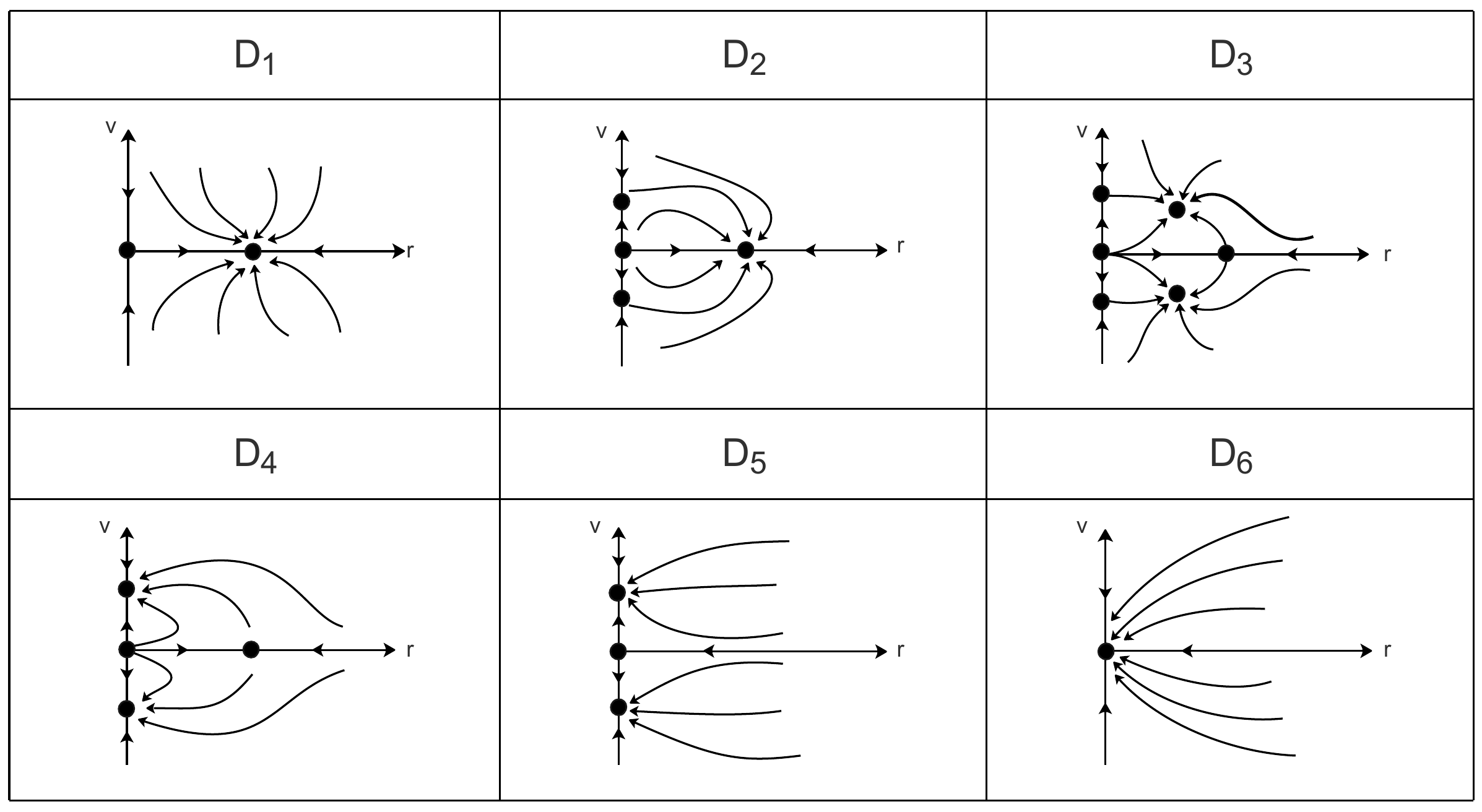} \label{fig32}                  
	\end{minipage}
	}                                               
	\caption{(a) Bifurcation sets in $(\alpha_1,\alpha_2)$ plane. (b) Phase portraits in $D_1-D_6$.}                                              
\end{figure}
The $(\alpha_1,\alpha_2)$ plane is divided into six regions $D_1-D_6$ by the branch line $T_1-T_4$. The detailed dynamics in $D_1$-$D_6$ near the origin have been shown in Figure \ref{fig32}. 
%
There exists  one stable spatially homogeneous periodic solution in $D_1$ and $D_2$, two stable spatially inhomogeneous periodic solutions in $D_3$, two stable  spatially steady states in $D_4$ and $D_5$, one stable spatially homogeneous steady state in $D_6$.

In this paper, we just exhibit the stable spatiotemporal and spatial patterns in $D_3$ and $D_5$.  
In Figure \ref{fig4}, parameters $(\alpha_1,\alpha_2) = (0.05,-0.33)\in D_3$, and the coexistence spatially inhomogeneous periodic solutions are find.
In Figure \ref{fig5}, parameters $(\alpha_1,\alpha_2) = (-0.1,-0.4)\in D_5$, and two spatially inhomogeneous steady state solutions are shown.
\begin{figure}[htb]
	\centering                           
	\subfigure[$u(t,x)$]{       
		\begin{minipage}{0.48\linewidth} 
			\centering                                      
			\includegraphics[ scale=0.25]{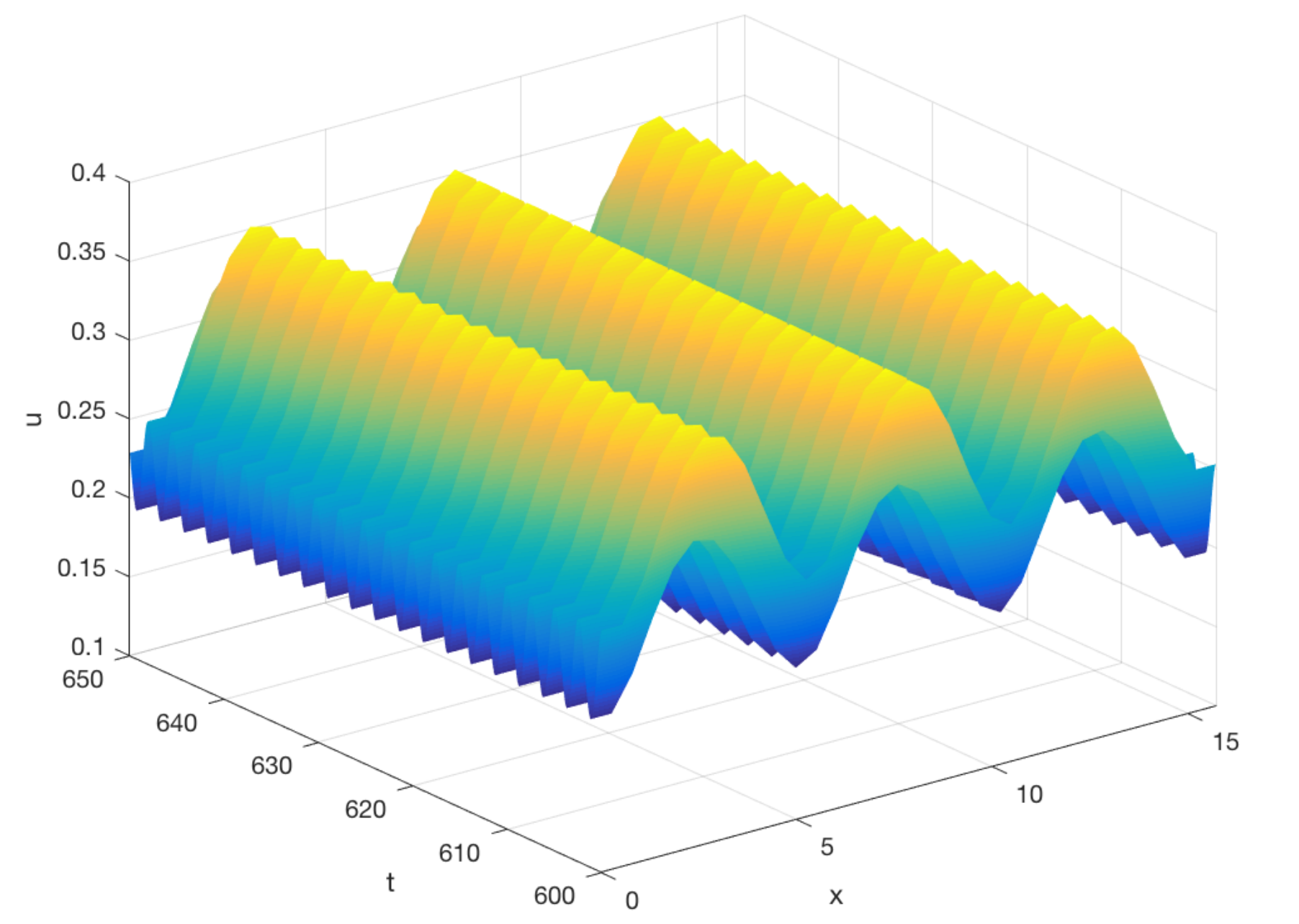}            
	\end{minipage}}
	\subfigure[$v(t,x)$]{                  
		\begin{minipage}{0.48\linewidth} 
			\centering                                     
			\includegraphics[scale=0.25]{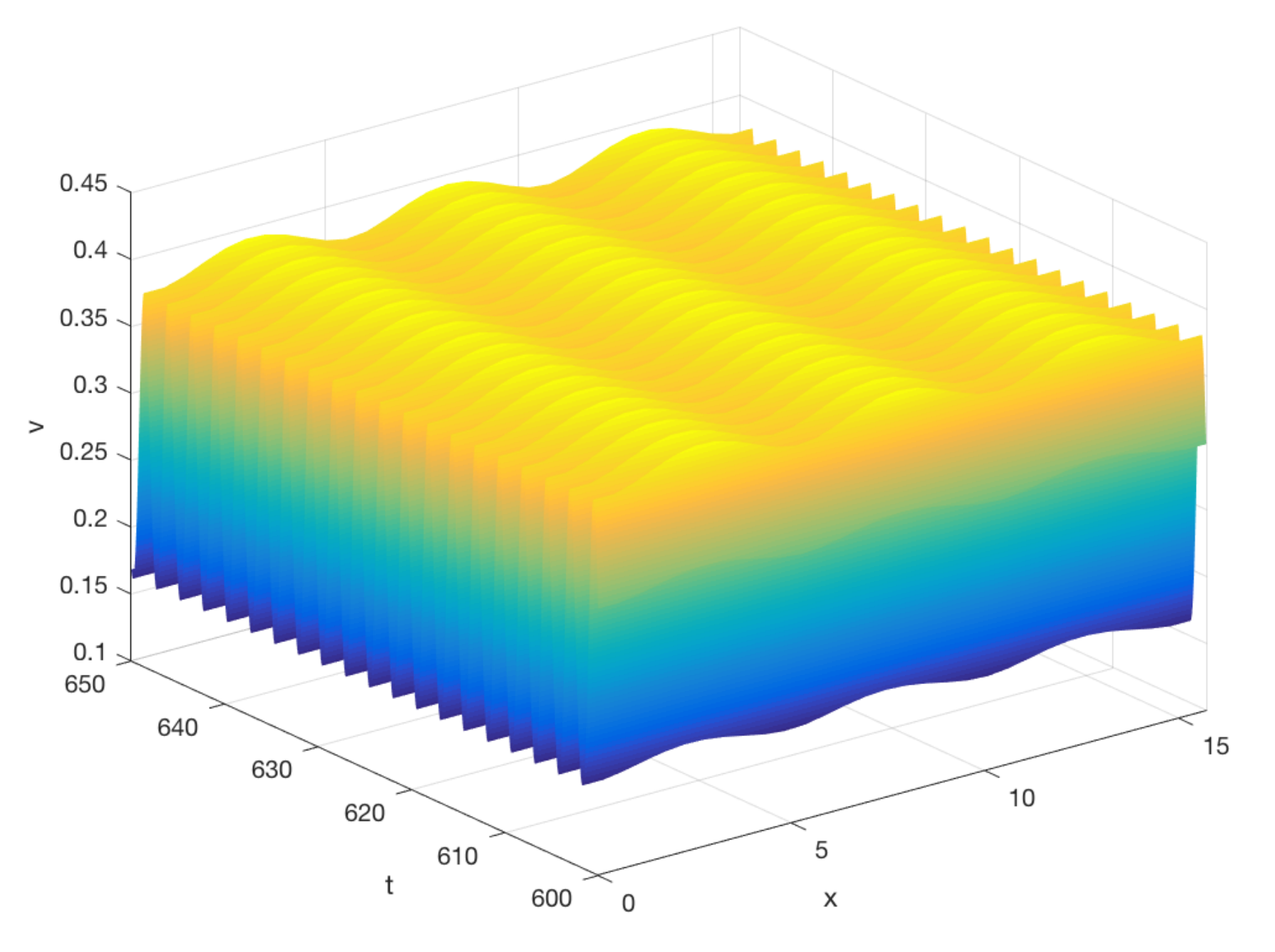}                
		\end{minipage}
	}
	
	\subfigure[$u(t,x)$]{       
		\begin{minipage}{0.48\linewidth} 
			\centering                                      
			\includegraphics[scale=0.25]{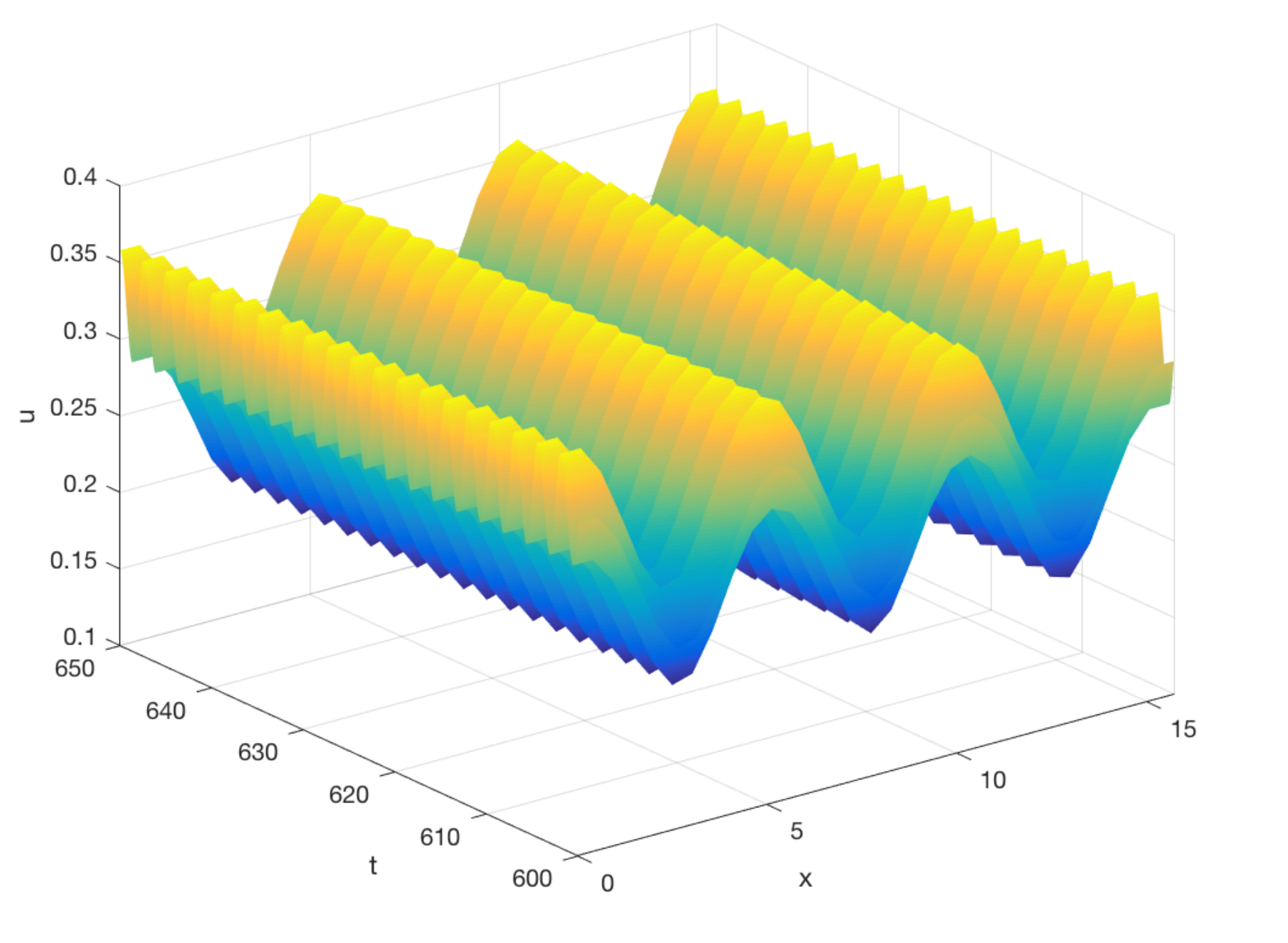}            
	\end{minipage}}
	\subfigure[$v(t,x)$]{                  
		\begin{minipage}{0.48\linewidth} 
			\centering                                     
			\includegraphics[scale=0.25]{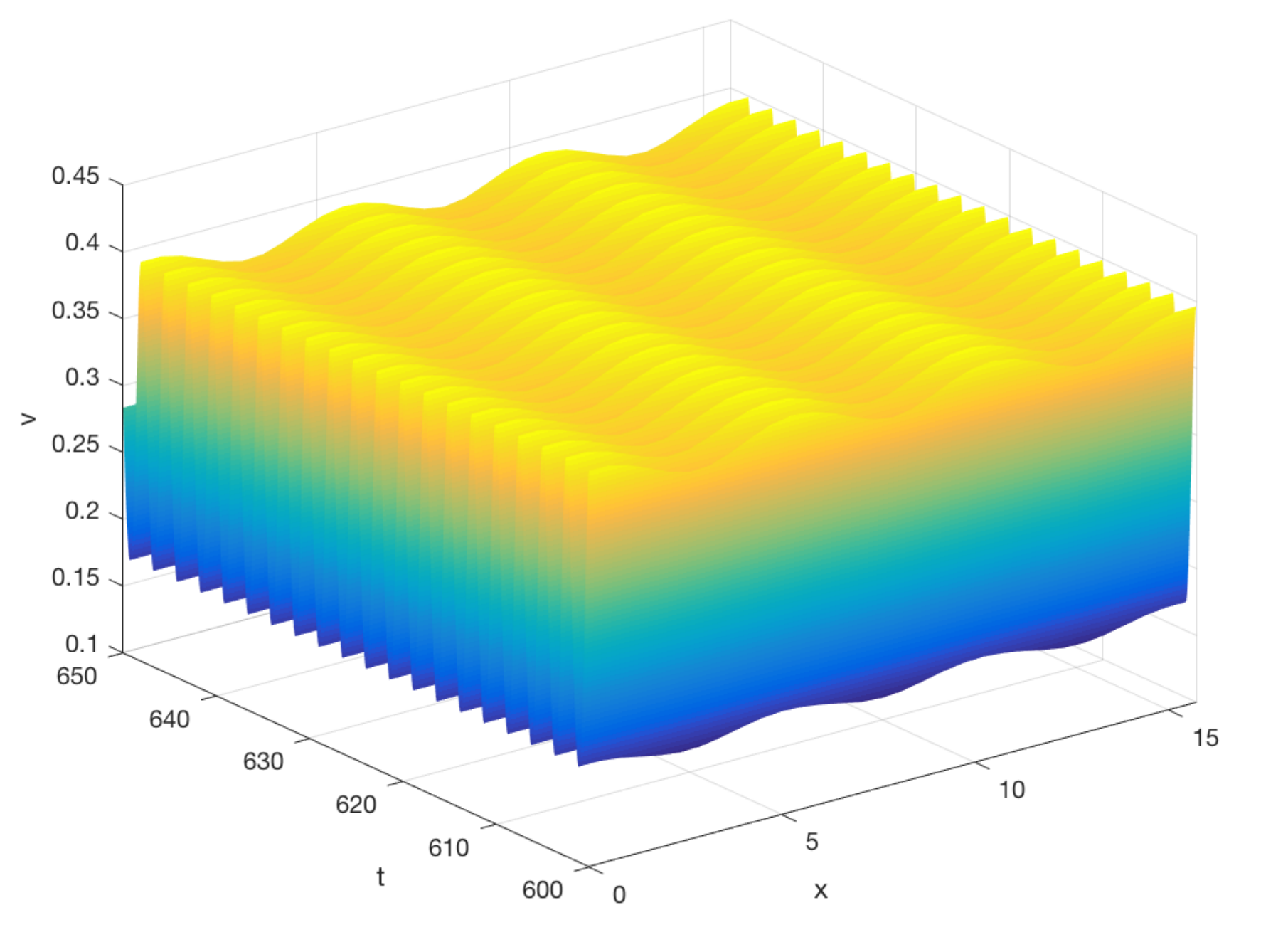}                
		\end{minipage}
	}
	\caption{Two spatially inhomogeneous periodic solutions coexist in $D_3$, with $(\alpha_1,\alpha_2) = (0.05,-0.33)$. (a),(b) are the solutions $u(t,x),v(t,x)$ of \eqref{eq5.1} with the initial value functions $(\varphi(t,x),\psi(t,x))=(u_0+0.005\sin x,v_0+0.001\sin x)$. (c),(d) are the solutions $u(t,x),v(t,x)$ of \eqref{eq5.1} with the initial value functions $(\varphi(t,x),\psi(t,x))=(u_0-0.005\sin x,v_0-0.001\sin x)$.} \label{fig4}                                                 
\end{figure}
\begin{figure}[htb]
	\centering                           
	\subfigure[$u(t,x)$]{       
		\begin{minipage}{0.48\linewidth} 
			\centering                                      
			\includegraphics[scale=0.25]{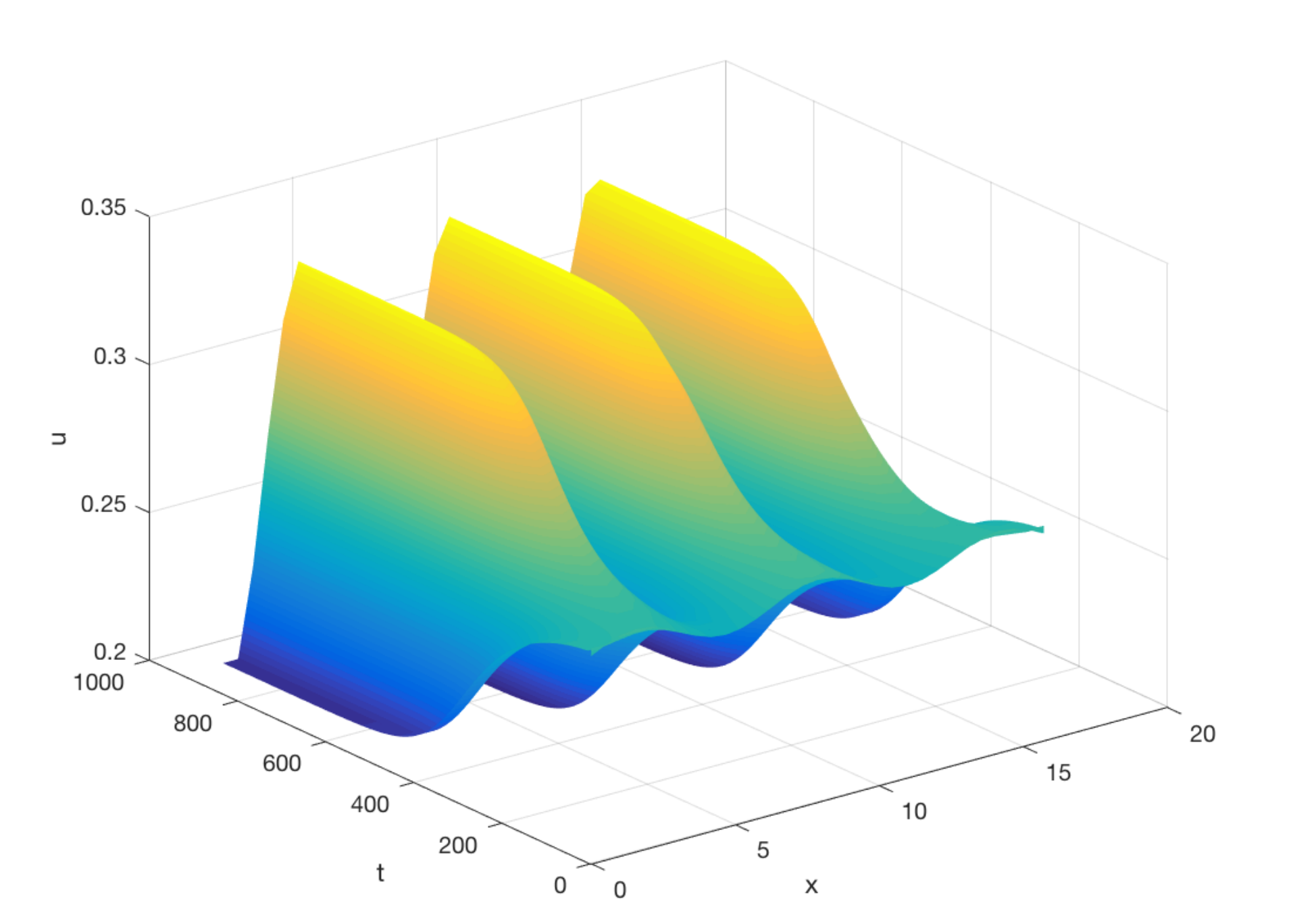}            
	\end{minipage}}
	\subfigure[$v(t,x)$]{                  
		\begin{minipage}{0.48\linewidth} 
			\centering                                     
			\includegraphics[scale=0.25]{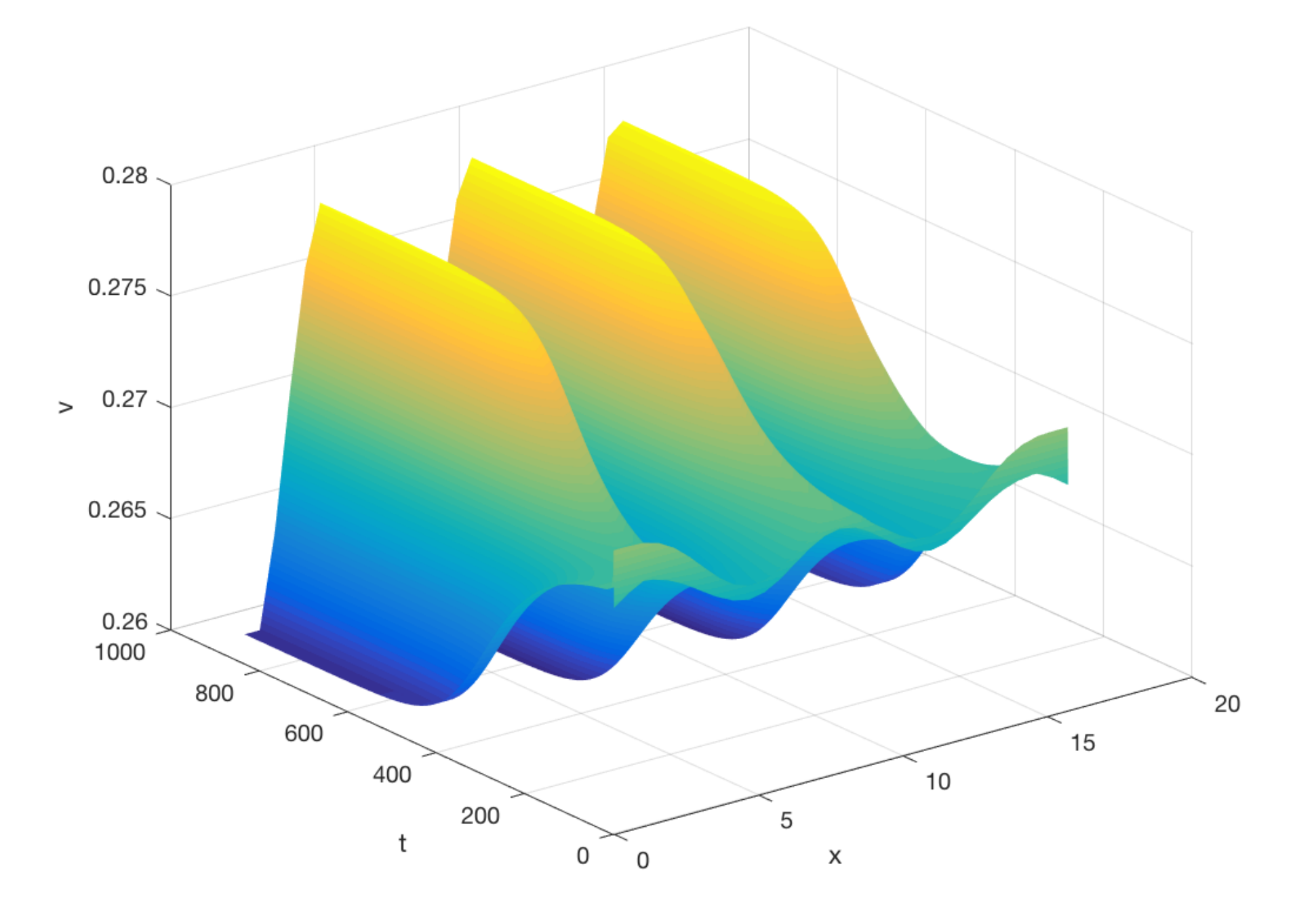}                
		\end{minipage}
	}
	
	\subfigure[$u(t,x)$]{       
		\begin{minipage}{0.48\linewidth} 
			\centering                                      
			\includegraphics[scale=0.25]{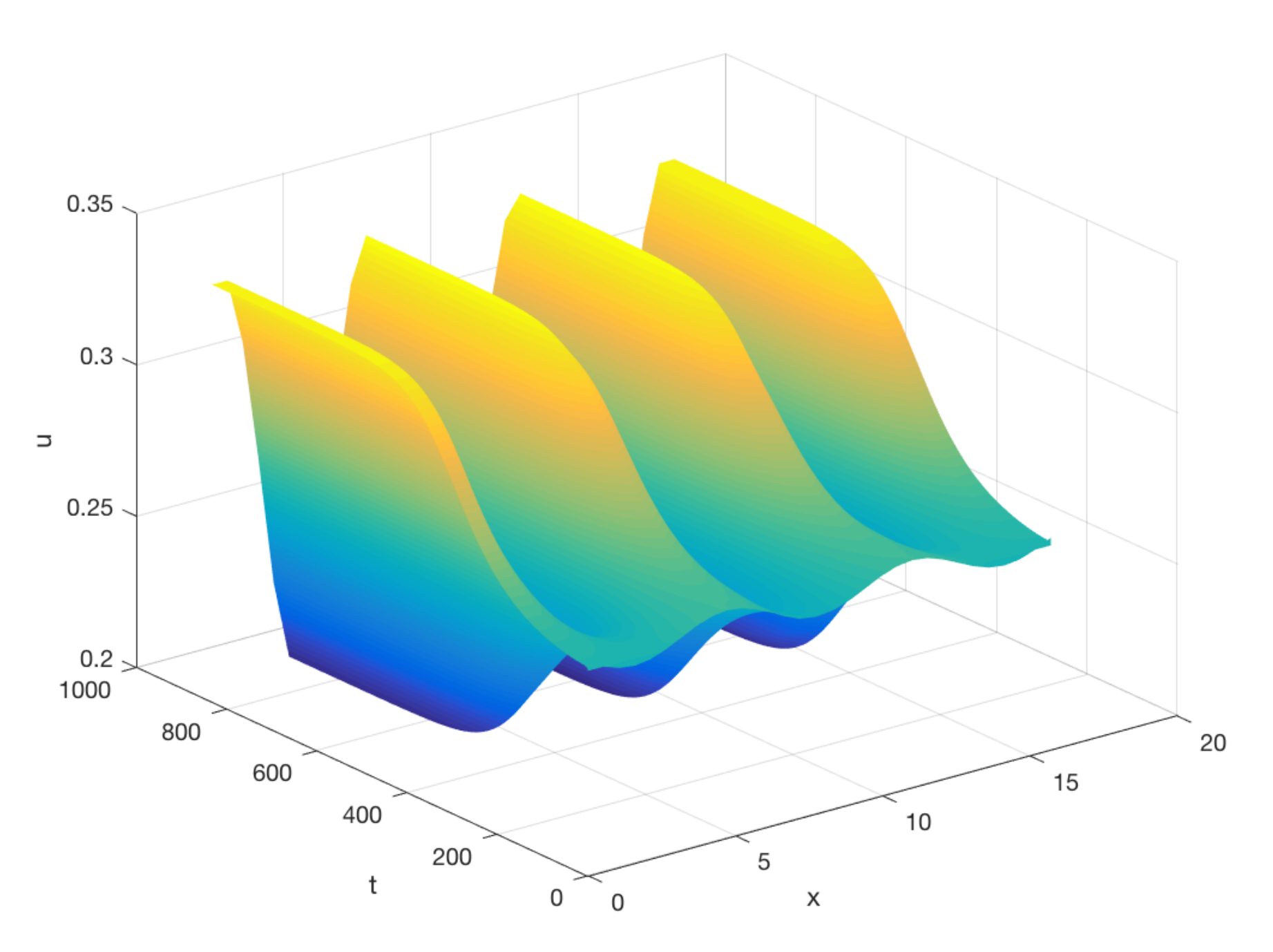}            
	\end{minipage}}
	\subfigure[$v(t,x)$]{                  
		\begin{minipage}{0.48\linewidth} 
			\centering                                     
			\includegraphics[scale=0.25]{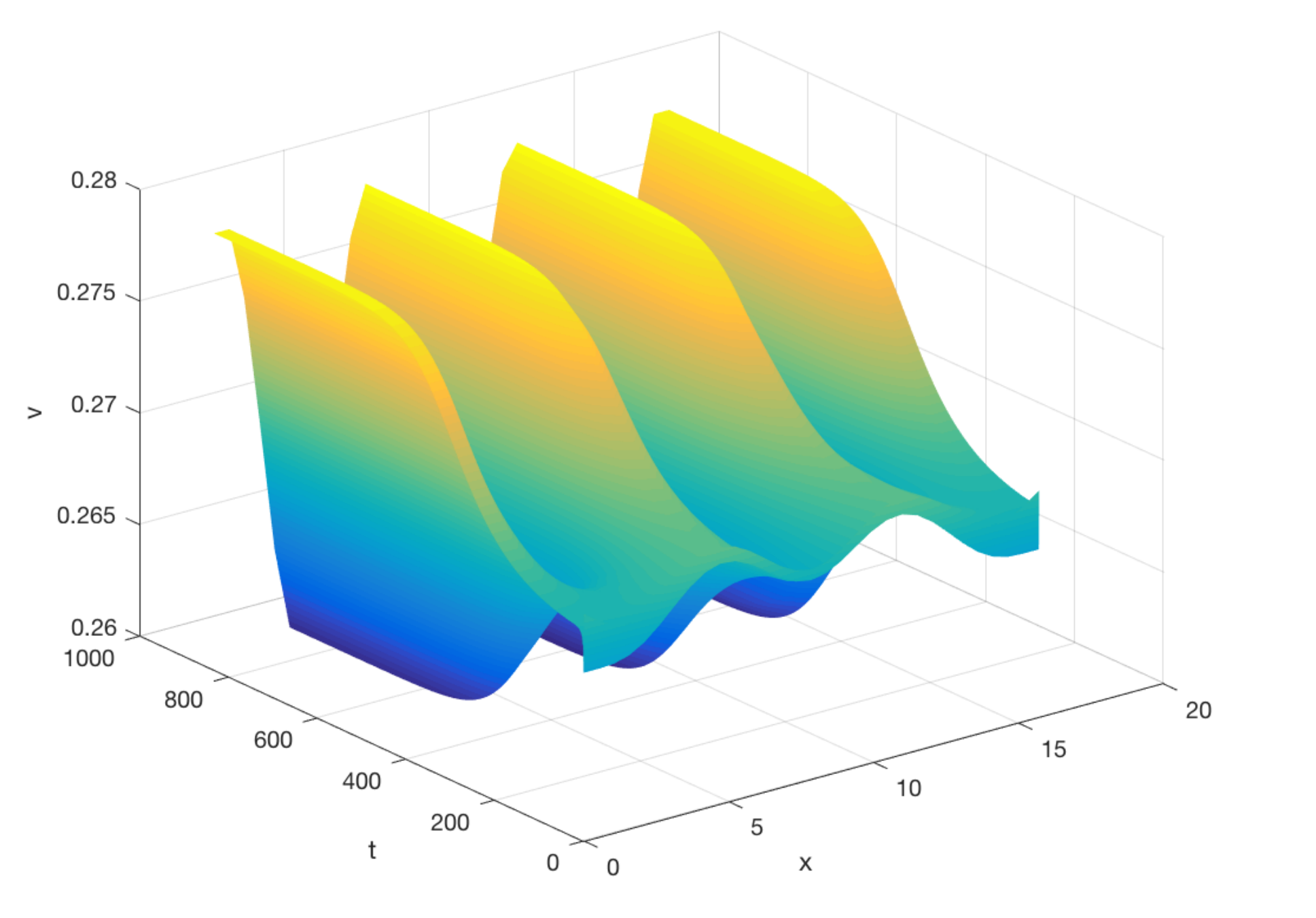}                
		\end{minipage}
	}
	\caption{Two spatially inhomogeneous steady state solutions coexist in $D_5$, with $(\alpha_1,\alpha_2) = (-0.1,-0.4)$.  (a),(b) are the solutions $u(t,x),v(t,x)$ of \eqref{eq5.1} with the initial value functions $(\varphi(t,x),\psi(t,x))=(u_0+0.005\sin x,v_0+0.001\sin x)$. (c),(d) are the solutions $u(t,x),v(t,x)$ of \eqref{eq5.1} with the initial value functions $(\varphi(t,x),\psi(t,x))=(u_0-0.005\sin x,v_0-0.001\sin x)$.} \label{fig5}                                                 
\end{figure}
\section{Conclusions}
Rigorous calculation formulas of the normal forms for PFDEs with a Hopf-zero singular point are given, for the first time. 
We decompose the enlarged phase space directly by means of the phase space decomposition of the corresponding RFDEs.
The necessary condition $U_2^2(z,0)\in V_2^3(\mathcal{Q}_1)$ has been replaced by a weaker one \eqref{U22=}, which more convenient for application.

The concrete formulas of eigenfunctions and the normal forms for Turing-Hopf bifurcation are present. Moreover, the corresponding dynamic  behaviors of the original system \eqref{A} are studied by the further analysis of the truncated  normal forms \eqref{eqNF3}. These reveals that Turing-Hopf bifurcation is one of the mechanism for the formation of the spatio-temporal pattern, which is oscillate in both space and time.

As an illustration of this calculation formula and the Turing-Hopf theory, Holling-Tanner model with a discrete time delay is investigated. The results of numerical simulations have shown that the spatio-temporal patterns can be self-organized when the parameters in some regions close to Turing-Hopf bifurcation points.

The greatest advantage of the formulas \eqref{eqNF3} is that they are convenient for computer processing, since they are only based on simple matrix operations and derivations.  In other words, 
we only need to enter the different models into the program, then the second and third order  coefficients in normal form and the planner system \eqref{rv} can be computed directly by the Matlab.
	\section*{Appendix A}
	Doing the second-order Fr\'{e}chet derivative of $F_{0}(\alpha,\varphi)$  and decomposing $U^t\in\mathcal{C}_0^1$ as in \eqref{U^{t}}. With the help of MATLAB, we obtain the coefficient vectors $F_{\alpha_iz_j},$ $F_{y_i(\theta)z_j},$ $F_{mnk}$ of $F_{0}^{(2)}(z,y,\alpha)$ are
	\begin{equation*}
	\begin{split}
	&F_{\alpha_{i}z_1} =2\,\left[\frac{\partial}{\partial \mu_i}A(\mu_0)\phi_1(0)+\frac{\partial}{\partial \mu_i}B(\mu_0)\phi_1(-1)\right],\hspace{3.35cm} (i=1,2),\\
	&F_{\alpha_{i}z_{2}} =2\,[-\frac{n_2^2}{l^2}\frac{\partial}{\partial \mu_i}D(\mu_0)\phi_2(0)+\frac{\partial}{\partial \mu_i}A(\mu_0)\phi_2(0)+\frac{\partial}{\partial \mu_i}\textbf{}(\mu_0)\phi_2(-1)],\quad (i=1,2),\\
	&F_{y_1(0)z_1} \;\;= 2( F_{uu} \!+\! F_{uv} k_{1} \!+\! F_{uu_{\tau}} e^{\mathrm{-i}\omega_{0}} \!+\! F_{uv_{\tau}} k_{1} e^{\mathrm{-i}\omega_{0}}),\hspace{2cm}\\
	&F_{y_1(-1)z_1} = 2( F_{uu_{\tau}} \!+\! F_{vu_{\tau}} k_{1} \!+\!  F_{u_{\tau}u_{\tau}} e^{\mathrm{-i}\omega_{0}} \!+\! F_{u_{\tau}v_{\tau}} k_{1} e^{\mathrm{-i}\omega_{0}}),\\
	&F_{y_2(0)z_1} \;\;= 2( F_{uv} \!+\!  F_{vv} k_{1} \!+\! F_{vu_{\tau}} e^{\mathrm{-i}\omega_{0}} \!+\! F_{vv_{\tau}} k_{1} e^{\mathrm{-i}\omega_{0}}),\\
	&F_{y_2(-1)z_1} =  2(F_{uv_{\tau}} \!+\! F_{vv_{\tau}} k_{1} \!+\! F_{u_{\tau}v_{\tau}} e^{\mathrm{-i}\omega_{0}} \!+\!  F_{v_{\tau}v_{\tau}} k_{1} e^{\mathrm{-i}\omega_{0}}),\\
	&F_{y_1(0)z_{2}} \;\;= 2( F_{uu} \!+\!  F_{uu_{\tau}} \!+\! F_{uv} k_{3} \!+\! F_{uv_{\tau}} k_{3}),\\	
	&F_{y_1(-1)z_{2}} = 2(F_{uu_{\tau}} \!+\! F_{u_{\tau}u_{\tau}} \!+\! F_{u_{\tau}v_{\tau}} k_{3} \!+\! F_{vu_{\tau}} k_{3}),\hspace{4.7cm}\\
	&F_{y_2(0)z_{2}} \;\;=2( F_{uv} \!+\! F_{vu_{\tau}} \!+\!  F_{vv} k_{3} \!+\! F_{vv_{\tau}} k_{3} ),\\
	&F_{y_2(-1)z_{2}} =2( F_{uv_{\tau}} \!+\! F_{u_{\tau}v_{\tau}} \!+\! F_{vv_{\tau}} k_{3} \!+\!  F_{v_{\tau}v_{\tau}} k_{3}),
	\end{split}
	\end{equation*}
	and
	\begin{equation*}
	\begin{split}	
	&F_{200}  = F_{uu} \!\!+\! F_{vv} k_{1}^2\!+\!    F_{u_{\tau}u_{\tau}} e^{\mathrm{-2i}\omega_{0}} \!+\! F_{v_{\tau}v_{\tau}} k_{1}^2 e^{\mathrm{-2i}\omega_{0}} +2( F_{uv} k_{1} \!+\! F_{uu_{\tau}} e^{\mathrm{-i}\omega_{0}}
	\\&\hspace{1.2cm}\!+\! F_{uv_{\tau}} k_{1} e^{\mathrm{-i}\omega_{0}} \!+\! F_{vu_{\tau}} k_{1} e^{\mathrm{-i}\omega_{0}}
	\!+\! F_{vv_{\tau}} k_{1}^2 e^{\mathrm{-i}\omega_{0}}
	\!+\! F_{u_{\tau}v_{\tau}} k_{1} e^{\mathrm{-2i}\omega_{0}}),\\
	&F_{110}  =\! 2[ F_{uu} \!+\!  F_{vv} k_{1}\bar{k}_{1}\!+\!  F_{u_{\tau}u_{\tau}} \!+\!  F_{v_{\tau}v_{\tau}} k_{1}\bar{k}_{1}\!+\! F_{uv} (k_{1}+\bar{k}_1) \!+\! F_{uu_{\tau}} (e^{\mathrm{\!-\!i}\omega_{0}} \!\!+\!  e^{\mathrm{i}\omega_{0}})\\&\hspace{1.2cm}+F_{uv_{\tau}}( k_{1} e^{\mathrm{-i}\omega_{0}} +\bar{k}_{1}e^{\mathrm{i}\omega_{0}} )
	\!+\! F_{vu_{\tau}} (\bar{k}_{1}e^{\mathrm{-i}\omega_{0}}+ k_{1} e^{\mathrm{i}\omega_{0}})\\&\hspace{1.2cm}+ F_{vv_{\tau}} k_{1}\bar{k}_{1}( e^{\mathrm{-i}\omega_{0}}\!+\!e^{\mathrm{i}\omega_{0}}) 
	\!+\! F_{u_{\tau}v_{\tau}}( k_{1}+\bar{k}_{1})],\\
	&F_{101}  = 2 [F_{uu} +F_{vv} k_{1} k_{3}+ F_{u_{\tau}u_{\tau}} e^{\mathrm{-i}\omega_{0}} + F_{v_{\tau}v_{\tau}} k_{1} k_{3} e^{\mathrm{-i}\omega_{0}}+ F_{uv} (k_{1} +k_{3})\\&\hspace{1.2cm}
	+ F_{uu_{\tau}} (1+ e^{\mathrm{-i}\omega_{0}})+F_{uv_{\tau}}( k_{3}+k_{1} e^{\mathrm{-i}\omega_{0}})
	+  F_{vu_{\tau}}( k_{1} + k_{3} e^{\mathrm{-i}\omega_{0}})\\&\hspace{1.2cm}
	+F_{vv_{\tau}} k_{1} k_{3}(1+ e^{\mathrm{-i}\omega_{0}}) + F_{u_{\tau}v_{\tau}} (k_{1}+ k_{3}) e^{\mathrm{-i}\omega_{0}}],\\
	&F_{002}  =\! F_{uu} +  F_{vv} k_{3}^2 + F_{u_{\tau}u_{\tau}}\! \!+  F_{v_{\tau}v_{\tau}} k_{3}^2 + 2( F_{uv} k_{3} + F_{uu_{\tau}} \!+\! F_{uv_{\tau}} k_{3}  +  F_{vu_{\tau}} k_{3}
	\\&\hspace{1.2cm} + F_{vv_{\tau}} k_{3}^2 +  F_{u_{\tau}v_{\tau}} k_{3}),\\
	&F_{020} = \overline{F_{200}},\\\hspace{1cm}
	&F_{011}  = \overline{F_{101}}.\\
	\end{split}
	\end{equation*}
	Here $u,v,u_{\tau},v_{\tau}$ are denote as the simplified form of $u(t),v(t),u(t-1),v(t-1)$, respectively. $F_{uu}:=\frac{\partial ^2}{\partial u^2}F(\mu_0,0)$, and so forth.
	
	In a similar ways, we obtain the requisite coefficient vectors $F_{mnk}$ of $F_0^{(3)}(z,0,0)$ by Matlab:
	\begin{equation*}
	\begin{aligned}
	F_{210} =& 3[F_{uuu}+F_{uuv}(2 k_{1}+\bar{k}_1)+F_{uuu_{\tau}}(2 e^{\mathrm{-i}\omega_{0}}+e^{\mathrm{i}\omega_{0}})	+  F_{uvv} k_1(2 \bar{k}_{1}+k_{1})~~~ \\&+  F_{uuv_{\tau}} (2 k_{1} e^{\mathrm{-i}\omega_{0}}+\bar{k}_{1}e^{\mathrm{i}\omega_{0}} )
	+ 2F_{uvu_{\tau}} (k_{1} e^{\mathrm{-i}\omega_{0}}+k_{1} e^{\mathrm{i}\omega_{0}}+ \bar{k}_{1}e^{\mathrm{-i}\omega_{0}})\\&+  2F_{uvv_{\tau}}k_1 (\bar{k}_{1}e^{\mathrm{-i}\omega_{0}} + \bar{k}_{1}e^{\mathrm{i}\omega_{0}}+k_{1} e^{\mathrm{-i}\omega_{0}})+  F_{uu_{\tau}u_{\tau}}(e^{\mathrm{-2i}\omega_{0}}+2)\\&+ 2F_{uu_{\tau}v_{\tau}}(k_{1}+\bar{k}_{1}+k_{1}e^{\mathrm{-2i}\omega_{0}})+  F_{uv_{\tau}v_{\tau}}k_1(2\bar{k}_{1}+k_{1} e^{\mathrm{-2i}\omega_{0}})\\&+  F_{vvv} k_{1}^2\bar{k}_{1}+  F_{vvu_{\tau}} k_1(2  \bar{k}_{1}e^{\mathrm{-i}\omega_{0}}+k_{1} e^{\mathrm{i}\omega_{0}})+  F_{vvv_{\tau}} k_1^2\bar{k}_{1}(2 e^{\mathrm{-i}\omega_{0}} + e^{\mathrm{i}\omega_{0}} ) \\&+  F_{vu_{\tau}u_{\tau}} (2 k_{1}+\bar{k}_{1}e^{\mathrm{-2i}\omega_{0}} )+2  F_{vu_{\tau}v_{\tau}} k_1(\bar{k}_{1}+\bar{k}_{1}e^{\mathrm{-2i}\omega_{0}}+k_{1})\\&+  F_{vv_{\tau}v_{\tau}} k_1^2\bar{k}_{1} (2 +e^{\mathrm{-2i}\omega_{0}})+ F_{u_{\tau}u_{\tau}u_{\tau}} e^{\mathrm{-i}\omega_{0}}+  F_{u_{\tau}u_{\tau}v_{\tau}} (2 k_{1} e^{\mathrm{-i}\omega_{0}}\\&+ \bar{k}_{1}e^{\mathrm{-i}\omega_{0}})+  F_{u_{\tau}v_{\tau}v_{\tau}}k_1  e^{\mathrm{-i}\omega_{0}}(2\bar{k}_{1} +k_{1})+ F_{v_{\tau}v_{\tau}v_{\tau}}k_{1}^2\bar{k}_1e^{\mathrm{-i}\omega_{0}}],\\
	F_{102} =& 3 [F_{uuu} \!+\! F_{uuv} ( k_{1} \!+\! 2  k_{3}) \!+\! F_{uuu_{\tau}} (e^{\mathrm{-i}\omega_{0}} \!+\! 2) \!+\! F_{uuv_{\tau}} (2  k_{3} \!+\!  k_{1} e^{\mathrm{-i}\omega_{0}})\\& \!+\! F_{uvv} ( k_{3}^2 \!+\! 2  k_{1}  k_{3}) +2 F_{uvu_{\tau}} ( k_{1} \!+\!  k_{3} \!+\!  k_{3} e^{\mathrm{-i}\omega_{0}}) \\&\!+\! 2 F_{uvv_{\tau}} ( k_{1}  k_{3} \!+\!  k_{3}^2 \!+\!  k_{1}  k_{3} e^{\mathrm{-i}\omega_{0}}) \!+\! F_{uu_{\tau}u_{\tau}} (2 e^{\mathrm{-i}\omega_{0}} \!+\! 1) \\&+2 F_{uu_{\tau}v_{\tau}} ( k_{3} \!+\!  k_{1} e^{\mathrm{-i}\omega_{0}} \!+\!  k_{3} e^{\mathrm{-i}\omega_{0}}) \!+\! F_{uv_{\tau}v_{\tau}}  k_{3}( k_{3} \!+\! 2  k_{1} e^{\mathrm{-i}\omega_{0}} ) \\&\!+\!  F_{vvv}  k_{1}  k_{3}^2 \!+\! F_{vvu_{\tau}} k_{3}(k_{3}e^{\mathrm{-i}\omega_{0}}   \!+\! 2  k_{1} ) \!+\! F_{vvv_{\tau}} k_{1}  k_{3}^2 (2  \!+\!   e^{\mathrm{-i}\omega_{0}}) \\&\!+\! F_{vu_{\tau}u_{\tau}} ( k_{1} \!+\! 2  k_{3} e^{\mathrm{-i}\omega_{0}}) \!+\! 2F_{vu_{\tau}v_{\tau}}  k_{3}( k_{1}  \!+\!  k_{3} e^{\mathrm{-i}\omega_{0}} \!+\!  k_{1} e^{\mathrm{-i}\omega_{0}}) \\&\!+\! F_{vv_{\tau}v_{\tau}}  k_{1}  k_{3}^2( 1\!+\! 2  e^{\mathrm{-i}\omega_{0}}) \!+\! F_{u_{\tau}u_{\tau}u_{\tau}} e^{\mathrm{-i}\omega_{0}} \!+\! F_{u_{\tau}u_{\tau}v_{\tau}} ( k_{1} e^{\mathrm{-i}\omega_{0}} \!+\! 2  k_{3} e^{\mathrm{-i}\omega_{0}}) \\&\!+\! F_{u_{\tau}v_{\tau}v_{\tau}} k_{3}e^{\mathrm{-i}\omega_{0}} (k_{3} \!+\! 2  k_{1} ) \!+\! F_{v_{\tau}v_{\tau}v_{\tau}}  k_{1}  k_{3}^2 e^{\mathrm{-i}\omega_{0}}],\\
	\end{aligned}
	\end{equation*}
	\begin{equation*}
	\begin{aligned}
	F_{111} =&  6\{ F_{uuu}  +  F_{uuv} ( k_{1} + k_{3} +  \bar{k}_{1})  +  F_{uuu_{\tau}} ( e^{\mathrm{-i}\omega_{0}} +  e^{\mathrm{i}\omega_{0}} + 1)  \\&+  F_{uuv_{\tau}} ( k_{3} +  k_{1} e^{\mathrm{-i}\omega_{0}} +  \bar{k}_{1}e^{\mathrm{i}\omega_{0}} )  +  F_{uvv} ( k_{1}\bar{k}_{1} +  k_{1} k_{3}+  k_{3} \bar{k}_{1}) \\& +  F_{uvu_{\tau}} [k_{1}(1+e^{\mathrm{i}\omega_{0}}) + \bar{k}_{1}(1+e^{\mathrm{-i}\omega_{0}} )+ k_{3} (e^{\mathrm{-i}\omega_{0}}+e^{\mathrm{i}\omega_{0}})]  \\& +  F_{uvv_{\tau}}[ k_{1} k_{3}(1+ e^{\mathrm{-i}\omega_{0}}) + k_{1}\bar{k}_{1}( e^{\mathrm{-i}\omega_{0}} + e^{\mathrm{i}\omega_{0}} ) +\bar{k}_{1} k_{3} (1+ e^{\mathrm{i}\omega_{0}}) ] \\& +  F_{uu_{\tau}u_{\tau}} ( e^{\mathrm{-i}\omega_{0}} +  e^{\mathrm{i}\omega_{0}} +1 )  +  F_{uu_{\tau}v_{\tau}} [k_{1}(1 +  e^{\mathrm{-i}\omega_{0}})+ \bar{k}_{1}(1 + e^{\mathrm{i}\omega_{0}} ) \\& + k_{3}( e^{\mathrm{-i}\omega_{0}} + e^{\mathrm{i}\omega_{0}})]  +  F_{uv_{\tau}v_{\tau}} ( k_{1}\bar{k}_{1} + \bar{k}_{1} k_{3} e^{\mathrm{i}\omega_{0}}  +  k_{1} k_{3} e^{\mathrm{-i}\omega_{0}})  \\&+   F_{vvv}  k_{1}\bar{k}_{1}k_{3}  +  F_{vvu_{\tau}} ( k_{1}\bar{k}_{1} +  \bar{k}_{1} k_{3} e^{\mathrm{-i}\omega_{0}} +  k_{1} k_{3} e^{\mathrm{i}\omega_{0}})  \\&+  F_{vvv_{\tau}} k_{1}\bar{k}_{1} k_{3} ( 1+  e^{\mathrm{-i}\omega_{0}}  +  e^{\mathrm{i}\omega_{0}} )  +  F_{vu_{\tau}u_{\tau}} ( k_{3} +  k_{1} e^{\mathrm{i}\omega_{0}} +  e^{\mathrm{-i}\omega_{0}} \bar{k}_{1})  \\&+  F_{vu_{\tau}v_{\tau}} [k_{1} k_{3}(1+ e^{\mathrm{i}\omega_{0}}) +  k_{1}\bar{k}_{1}(e^{\mathrm{-i}\omega_{0}} + e^{\mathrm{i}\omega_{0}})+\bar{k}_{1} k_{3} (1+ e^{\mathrm{-i}\omega_{0}} ) ] \\& +  F_{vv_{\tau}v_{\tau}} k_{1}\bar{k}_{1}k_{3}(  1 +  e^{\mathrm{-i}\omega_{0}}  + e^{\mathrm{i}\omega_{0}}) +  F_{u_{\tau}u_{\tau}u_{\tau}} +   F_{u_{\tau}u_{\tau}v_{\tau}} ( k_{1} +  k_{3} +  \bar{k}_{1})\\&   +  F_{u_{\tau}v_{\tau}v_{\tau}} ( k_{1}\bar{k}_{1} +  k_{1} k_{3} +  k_{3} \bar{k}_{1})  +  F_{v_{\tau}v_{\tau}v_{\tau}} k_{1}\bar{k}_{1}  k_{3}\},\\
	F_{003} =& F_{uuu} + 3F_{uuu_{\tau}} +3 F_{uu_{\tau}u_{\tau}} + F_{u_{\tau}u_{\tau}u_{\tau}} +3 F_{uuv} k_{3} + 3F_{uuv_{\tau}} k_{3} \qquad\\&+ 6F_{uvu_{\tau}} k_{3} + 6F_{uu_{\tau}v_{\tau}} k_{3} +3 F_{u_{\tau}u_{\tau}v_{\tau}} k_{3} + 3F_{vu_{\tau}u_{\tau}} k_{3} + 3F_{uvv} k_{3}^2 \\&+6 F_{uvv_{\tau}} k_{3}^2 +3 F_{uv_{\tau}v_{\tau}} k_{3}^2 + F_{vvv} k_{3}^3 +3 F_{u_{\tau}v_{\tau}v_{\tau}} k_{3}^2 + 3F_{vvu_{\tau}} k_{3}^2 \\&+3 F_{vvv_{\tau}} k_{3}^3 + 6F_{vu_{\tau}v_{\tau}} k_{3}^2 + 3F_{vv_{\tau}v_{\tau}} k_{3}^3 + F_{v_{\tau}v_{\tau}v_{\tau}} k_{3}^3.
	\end{aligned}
	\end{equation*}
	Here $F_{uuu}:=\frac{\partial ^3}{\partial u^3}F(\mu_0,0)$, and so forth.



\medskip
Received xxxx 20xx; revised xxxx 20xx.
\medskip

\end{document}